\documentclass[letterpaper]{scrartcl}

\usepackage{mathtools, amsfonts, amssymb, amsthm, graphicx, hyperref,float,subcaption,fancyhdr,titling,tensor,mathrsfs,xcolor,tikz,tkz-euclide}
\usepackage[inline]{enumitem}
\usetikzlibrary{decorations.pathmorphing}
\usepackage[margin=.7in]{geometry}
\usepackage[affil-it]{authblk}

\newtheorem{theorem}{Theorem}[section]
\newtheorem{lemma}[theorem]{Lemma}

\theoremstyle{definition}
\newtheorem{definition}[theorem]{Definition}
\newtheorem*{example}{Example}

\theoremstyle{remark}

\makeatletter
\newcommand{\problem}{\@ifstar{\problemStar}{\problemNoStar}}
\newcommand{\problemStar}[1][]{\vspace{2\baselineskip}\refstepcounter{problem}{\noindent\large \bfseries Problem~#1}\\}
\newcommand{\problemNoStar}{\vspace{2\baselineskip}\refstepcounter{problem}{\noindent\large \bfseries Problem~\arabic{set}-\arabic{problem}}\\}
\makeatother

\makeatletter
\renewcommand*\env@matrix[1][*\c@MaxMatrixCols c]{%
	\hskip -\arraycolsep
	\let\@ifnextchar\new@ifnextchar
	\array{#1}}
\makeatother

\title{A Convergence Rate for Extended-Source Internal DLA in the Plane} 
\author{David Darrow\thanks{Supported in part by NSF Grant DMS 1500771.}} 

\date{\today}

\newcommand{\unit}[1]{\mathbf{\hat{#1}}}

\newcommand{\eps}{\varepsilon}
\newcommand{\Eps}{\mathcal{E}}
\newcommand{\ol}[1]{{\overline{#1}}}

\newcommand{\op}{\operatorname}

\begin{document}
	
	\maketitle
	\thispagestyle{fancy}
	
	\begin{abstract}
		\noindent\textbf{Abstract.} 
	Internal DLA (IDLA) is an internal aggregation model in which particles perform random walks from the origin, in turn, and stop upon reaching an unoccupied site. Levine and Peres showed that, when particles start instead from fixed multiple-point distributions, the modified IDLA processes have deterministic scaling limits related to a certain obstacle problem. In this paper, we investigate the convergence rate of this ``extended source'' IDLA in the plane to its scaling limit. We show that, if $\delta$ is the lattice size, fluctuations of the IDLA occupied set are at most of order $\delta^{3/5}$ from its scaling limit, with probability at least $1-e^{-1/\delta^{2/5}}$.
	\end{abstract}
	
	
	
	
	
	
	
	
	
	
	
	\tableofcontents

	
	\section{Introduction}
	Internal Diffusion Limited Aggregation (IDLA) is a probabilistic growth process on the integer lattice $\mathbb{Z}^d$, first proposed by Meakin and Deutch \cite{doi:10.1063/1.451129} to model electro-chemical polishing. Namely, IDLA follows the growth of random sets $A(n)$; we set $A(1)=\{0\}$, and $A(n+1)$ is obtained by adding to $A(n)$ the point at which a centered simple random walk exits $A(n)$. With the right scaling, this process resembles a stream of particles from the origin barraging (and thus smoothing) the inner surface of an origin-centered sphere.
	
	In line with its applications in smoothing processes, the overall smoothness of IDLA has been an active area of investigation. Meakin and Deutch first studied this numerically, finding that variations of $A(n)$ from the smooth ball were of magnitude $\log n$ in dimension 2 \cite{doi:10.1063/1.451129}. Significant progress has also been made in proving these properties mathematically. In particular, Lawler, Bramson, and Griffeath \cite{lawler1992} proved that $A(n)$ approaches the ball of radius $\sqrt[d]{n/\omega_d}$---where $\omega_d$ is the volume of the $d$-dimensional unit sphere---almost certainly as $n$ increases. Several groups \cite{lawler1995,asselah:hal-00795850} also found convergence rates for this process. Most recently, Asselah and Gaudillière proved that the fluctuations away from the disk are bounded by $\log^2 n$ in dimension 2 and $\sqrt{\log n}$ in higher dimensions \cite{asselah2013, Asselah_2013}, and Jerison, Levine, and Sheffield independently proved a $\log n$ bound in dimension 2 and $\sqrt{\log n}$ in higher dimensions \cite{10.2307/23072157,jerison2013}. Asselah and Gaudillière also proved \emph{lower} bounds of $\sqrt{\log n}$ on the maximum fluctuations of IDLA \cite{asselah2011lower}, showing that the recently proved results for $d\geq 3$ are tight.
	
	Of considerable interest is the extended-source case of IDLA, wherein particles start from a fixed point distribution rather than all from the origin. This generalizes the applicability of IDLA to a much wider range of surfaces, allowing us to see how different geometries interact with this smoothing process. This question was originally investigated by Diaconis and Fulton \cite{stanford1991growth} in the context of a ``smash sum'' of two domains. Levine and Peres \cite{Levine_2010} reframed this notion as a generalized IDLA, proving deterministic scaling limits for a piecewise constant density $\sigma:\mathbb{R}^d\to\mathbb{Z}_{\geq 0}$ of starting points. It is worth noting, but beyond the scope of this paper, that another model with Poisson particle sources was proposed and studied by Gravner and Quastel \cite{10.2307/2691943}.
	
	In this paper, we investigate the convergence rate of the extended-source IDLA of Levine and Peres to its scaling limit in dimension 2, adapting the techniques of Jerison et al.~\cite{10.2307/23072157}. Under the additional assumptions that the initial mass distribution is ``concentrated'' (see Section \ref{background}) and the deterministic limit of its IDLA flow is smooth, we show that---if $\delta$ is the lattice size---the fluctuations of extended-source IDLA are of order $\delta^{3/5}$ or below, with probability at least $1-e^{-1/\delta^{2/5}}$.
	
	There are several major difficulties in extending the argument of \cite{10.2307/23072157} to a general source setting, and we introduce and apply some new technical tools to solve them. In particular, their proof relies heavily upon a specific formulation of the Poisson kernel, which only applies in the case of the disk. Although we are able to make use of the disk Poisson kernel in the first part of our proof, we replace it halfway through our paper with a more general formula---combining results on the discrete Green's function with a last-exit decomposition (both taken from \cite{lawler2010random}), we obtain an explicit formula for discrete Poisson kernels in general domains. This approach requires relatively fine control of the discrete Green's function in general domains; in fact, we obtain an $L^1$ convergence rate of the discrete Green's function to its continuum limit, of order $\delta^{3/5}$ (see Lemma \ref{harmoniclemma2-2}(c)); we have not seen this result in the literature before, and we imagine it may be useful in studying similar problems.
	
	Finally, we believe that our overall $\delta^{3/5}$ bound on the fluctuations of IDLA is non-optimal, and we discuss possibilities for improvement in Section \ref{conclusion}. However, we will see in a sequel to this paper that our bound is strong enough to prove weak scaling limits of the IDLA fluctuations themselves; indeed, this subsequent result requires a bound of order $\delta^{1/2+\eps}$, for any $\eps>0$. The question of weak scaling limits of the IDLA fluctuations has been investigated in the single-source case by Jerison, Levine, and Sheffield \cite{jerison2014}---more recently, Eli Sadovnik \cite{eli16} has shown scaling limits for extended-source fluctuations integrated against harmonic polynomials. In the sequel, we will seek to generalize Sadovnik's result and apply it to fluctuations ``through time'', to better understand the covariances between fluctuations at different times.
	
	In Sections 2 and 3, we will introduce our main result and provide a background on existing theory needed for our proof. The remaining sections are dedicated to the proof of Theorem \ref{biggun'}. Sections 4 and 5 set up the necessary theory; the former shows that an early point implies a similarly late point, and the latter shows that a late point implies a different, very early point. Section 6 combines these results in an iterative argument, recovering the full theorem.
	\section{Background on extended-source IDLA}\label{background}
We will focus on a specific sort of extended source---a \emph{concentrated mass distribution}---slightly narrowing the definitions introduced in \cite{Levine_2010} in order to capture scaling limits for the partially-completed process. We will give details on various extensions in the final section.

\begin{definition}\label{massdist}
	Let $D_0\subset\mathbb{R}^2$ be a compact, connected domain with smooth boundary, and fix $N\in\mathbb{Z}^{\geq 0}$ and $T_1,...,T_N\in\mathbb{R}^{\geq 0}$. For each $i=1,...,N$ and $s\in[0,T_i]$, suppose $Q^s_i\subset D_0$ satisfies the following properties:
	\begin{enumerate}
		\item $Q_i^s$ is a compact domain with $\op{Vol}(Q_i^s)=s$.
		\item $Q_i^s$ is bounded away from $\partial D_0$---that is, $Q_i^s\subset\subset\op{int}(D_0)$.
		\item $Q_i^{s}\subset Q_i^{s'}$ for $s\leq s'\leq T_i$.
		\item $\partial Q_i^s$ is rectifiable, with arclength bounded independently of $s$.
	\end{enumerate}
	Finally, set $T=\sum_kT_k$, and fix increasing functions $s_i:[0,T]\to[0,T_i]$ satisfying $\sum_ks_k(s)=s$ for all $s\in[0,T]$.
	The \emph{concentrated mass distribution} associated to the data $(D_0,\{T_i\},\{Q^s_i\},\{s_i\})$ is the map $\sigma_s:\mathbb{R}^2\to\mathbb{Z}^{\geq 0}$ defined by
	\[\sigma_s=\mathbf{1}_{D_0}+\sum_{i=1}^{N}\mathbf{1}_{Q^{s_i(s)}_i}.\]
\end{definition}
We can also consider infinite concentrated mass distributions, allowing $s\in[0,\infty)$; we define these by requiring that the restriction $s\in[0,T]$ is a finite concentrated mass distribution, and we suppose we have fixed some such $T\geq 0$.

Geometrically, a concentrated mass distribution is a set of domains $Q_i^{s_i}\subset \subset\op{int}(D_0)$ growing at a rate $s_i(s)$. There are several differences and restrictions of this definition as compared to that in \cite{Levine_2010}, which require comment. Most notably, we define the mass distribution to grow in time, as $s\mapsto\sigma_s$, such that we can study the partially complete process (i.e., that at $s<T$). After discretization, this will correspond to a prescribed order of IDLA sources; this allows us to relate the IDLA process to a smoothly growing deterministic set, to recover uniform bounds (in time) on the IDLA fluctuation, and, in the sequel, to study correlations between fluctuations at different points in space \emph{and} in time.

That the arclength of $\partial Q_i^s$ exists and is bounded uniformly allows us to discretize the process $s\mapsto\sigma_s$ in a natural way. That is, viewing $\sigma_s$ as a multiset and taking any $m\geq 1$ and any time $s_0\in[0,T)$, we can find the ``first'' point in $\frac{1}{m}\mathbb{Z}^2$ added to $\sigma_s$ after the time $s_0$; this procedure will give the overall ordering of IDLA sources. As such, this hypothesis could be weakened, so long as the discretization remained possible.

The requirement that $Q_i^s\subset\subset\op{int}(D_0)$ is necessary for the proof of Lemma \ref{harmoniclemma}(c), which in turn is necessary for the estimate \ref{harmoniclemma2}(c). In short, we make heavy use of discrete Poisson kernels on sets cut out by IDLA, and in particular, their pointwise convergence (away from the pole) to continuous Poisson kernels; this hypothesis guarantees a minimum distance between this pole and the source points, which in turn guarantees a strong rate of convergence of these Poisson kernels at source points.

As we will see, IDLA processes begun from finer and finer discretizations of a concentrated mass distribution approach a smooth, deterministic flow $s\mapsto D_s$, where the set $D_s$ is defined in terms of the Diaconis--Fulton ``smash'' sum, as defined in \cite{Levine_2010}:
\begin{definition}
	If $A,B\subset\frac{1}{m}\mathbb{Z}^2$, we define the discrete smash sum $A\oplus B$ as follows. Let $C_0=A\cup B$, and for each $x_i\in \{x_1,...,x_n\}=A\cap B$, start a simple random walk at $x_i$ and stop it upon exiting $C_{i-1}$. Let $y_i$ be its final position, and define $C_{i}=C_{i-1}\cup\{y_i\}$. Then $A\oplus B:=C_n$ is a random set.
	
	As proven in \cite{Levine_2010}, if we instead take domains $A,B\subset\mathbb{R}^2$, the smash sums $A_m\oplus B_m$ of
	\[A_m:=\frac{1}{m}\mathbb{Z}^2\cap A,\qquad B_m:=\frac{1}{m}\mathbb{Z}^2\cap B\]
	approach a deterministic limit, which we label $A\oplus B$. An example is pictured in Figure \ref{squarefig}.
\end{definition}
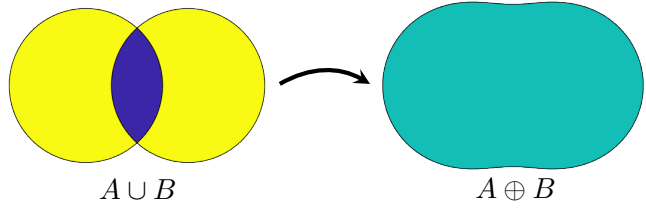
\begin{figure}[H]
	\centering
	\begin{tikzpicture}
		\node[anchor=south west,inner sep=0] (image) at (0,0) {\includegraphics[scale=.3]{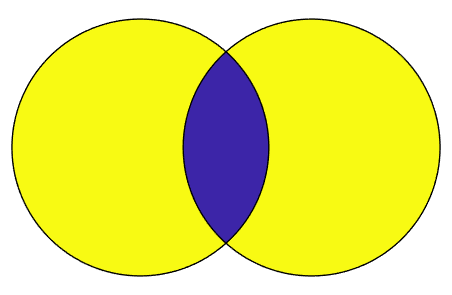}};
		\node[text width=2cm,rotate=0] at (2.3,-.2) {$A\cup B$};
		{}
		\path[ultra thick,->,>=stealth] (3.7,1.2) edge[bend left] (4.9,1.2);
		\node[anchor=south west,inner sep=0] (image) at (5,0) {\includegraphics[scale=.3]{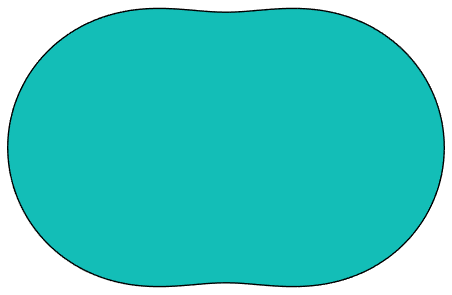}};
		\node[text width=2cm,rotate=0] at (7.3,-.2) {$A\oplus B$};
		{}
	\end{tikzpicture}%
	\caption{The smash sum $A\oplus B$ is the deterministic limit of an IDLA-type growth process starting from the sets $A$ and $B$, representing the dispersal of particles in $A\cap B$ (in dark blue above) to the edges of $A\cup B$ (in yellow above). In our setting, we see that IDLA also converges to an iterated smash sum.}\label{squarefig}
\end{figure}
Given a mass distribution $\sigma_s$ and using the notation of Definition \ref{massdist}, we define the sets $D_s$, $s\in[0,T]$ to be the smash sums
\[D_s=D_0\oplus Q_1^{s_1}\oplus\cdots\oplus Q_N^{s_N}.\]
Importantly, these are \emph{deterministic} sets, depending only on the mass distribution.

Visually, $D_s$ is a smooth outward flow from $D_0$, with $\op{Vol}(D_s)=s+\op{Vol}(D_0)$; we can think of $D_s$ as the result of allowing the mass at $\{Q_i^{s_i}\}$ to diffuse (in the sense of Brownian motion) to---and accumulate at---the edge of $D_0$. These sets satisfy the following key property:
\begin{lemma}\label{quaddomain}
	For $(D_s,\sigma_s)$ arising from a mass distribution, we have
	\[\int_{D_s}h = \int_{\mathbb{R}^2} h\sigma_s\]
	for any harmonic $h:D_s\to\mathbb{R}^2$.
\end{lemma}
This property, which identifies $D_s$ as a \emph{quadrature domain}, is well-known; for instance, see \cite{Sakai1984}.
\begin{example}\label{example1}
	Suppose that we take $D_0=B_1$ to be the unit disk, and we set $Q_i^{s}=B_{\sqrt{s/\pi}}$ for $s\in[0,\pi/4]$ and $i=1,...,N$, where $B_r=B_r(0)$ is generally the origin-centered disk of radius $r$. We further define $s_i(s)=s/N$, so that all sets $Q_i^{s}$ are growing at the same rate. Visually, particles are emanating evenly (with density $N$) from outwardly moving rings of radius $0\leq r\leq 1/2$, as shown in Figure \ref{diskfig}.
	
	Here, $T=N\pi/4$. From symmetry considerations, it is clear that $D_s=B_{\sqrt{1+s/\pi}}$ are outwardly expanding disks, as in the case of a point-source; we omit the proof here, as it is not critical to our results. The property of Lemma \ref{quaddomain} is simply the mean value property in this setting.
\end{example}
\begin{figure}[H]
	\centering
	\begin{tikzpicture}
		\node[anchor=south west,inner sep=0] (image) at (0,0) {\includegraphics[scale=.3]{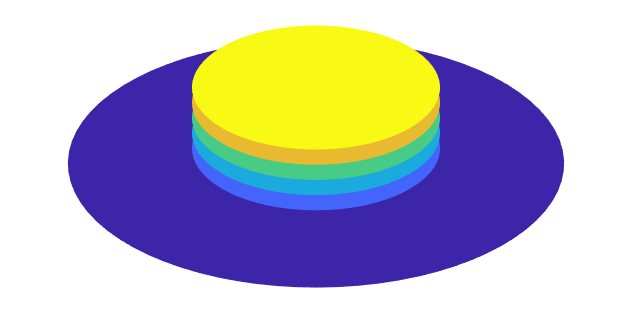}};
		\draw [->,
		line join=round,
		decorate, decoration={
			zigzag,
			segment length=15,
			amplitude=2,post=lineto,
			post length=2pt
		},ultra thick]  (5,1.4) -- (6.5,1.4);
		\node[anchor=south west,inner sep=0] (image) at (7,0) {\includegraphics[scale=.3]{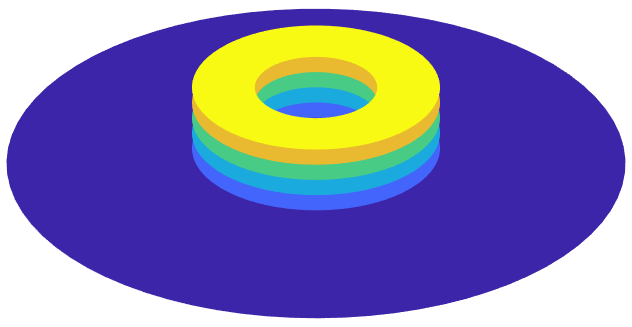}};
	\end{tikzpicture}%
	\caption{An illustration of the occupied set $D_s$ (dark blue) and the remaining source points $Q_i^{T_i}\setminus Q_i^{s_i}$ (multicolored) in Example \ref{example1}, at two different times. Here, our starting set $D_0=B_1$ is the unit disk, and our source points $Q_i^{s}=B_{\sqrt{s/\pi}}$ are identical radially-expanding disks within it (shown in 3D for visual clarity). From symmetry, we see that the occupied sets $D_s$ are also growing disks.}\label{diskfig}
\end{figure}
Next, we restrict attention to \emph{smooth} flows:
\begin{definition}
	The flow $D_s$ is \emph{smooth} if the flow $s\in[0,T]\mapsto D_s$ is a smooth isotopy from $D_{0}$ to $D_{T}$---that is, if the embeddings $\iota_s:D_s\hookrightarrow\mathbb{R}^2$ form a smooth homotopy from $\iota_0$ to $\iota_T$. Note that the disks of Example \ref{example1} form a smooth flow.
\end{definition}

To define discrete processes on a mass distribution, we first need to discretize the distribution itself; fortunately, there is a natural way to discretize any mass distribution. Fix an integer $m$, and note that $f(s)=\sum\nolimits_{\frac{1}{m}\mathbb{Z}^2}\sigma_s$ is an increasing, piecewise constant function of $s$. Let 
\[0= s_{m,0}<s_{m,1}<\cdots< s_{m,N'}=T\]
be a partition of $[0,T]$ with $\max_k(s_{m,k+1}-s_{m,k})=O(m^{-2})$ and such that $f$ is constant near $s_{m,k}$ for $s_{m,k}\neq 0,T$. Define the sequence $S_m=\{z_{m,1},...,z_{m,n_m}\}$ inductively as follows:
\begin{enumerate}
	\item Let $s_{m,n}$ be the smallest $s_{m,i}\geq s_{m,n-1}$ such that 
	\[\sum_{\frac{1}{m}\mathbb{Z}^2}(\sigma_{s_{m,n}}-\sigma_0)-\sum_{i<n} z_{m,i}>0.\]
	\item Choose $z_{m,n}\in\frac{1}{m}\mathbb{Z}^2$ such that $(\sigma_{s_{m,n}}-\sigma_0)(z_{m,n})$ exceeds the number of times $z_{m,n}$ occurs in $\{z_{m,i}\}_{i<n}$.
\end{enumerate}

Intuitively, we allow the sets $Q_i^{s_i}$ to expand a slight (i.e., $O(m^{-2})$) amount, and then we add all new points to the sequence of $z_{m,i}$. It is possible that multiple points may satisfy this condition for a given time---in the limit $m\to\infty$, the order in which these ``nearby'' points appear will not matter.

Given these sequences $S_m$, the (resolution $m$) \emph{internal DLA (IDLA)} associated to the mass distribution is the following process:
\begin{definition}[Internal DLA]
	Suppose we have a concentrated mass distribution with initial set $D_0$ giving rise to the sequences $S_m$. The IDLA $A_m(t)$ associated to the mass distribution is as follows. Define the initial set $A_m(0)=\frac{1}{m}\mathbb{Z}^2\cap D_0$. Then, for each $i\geq 1$, start a random walk at $z_{m,i}$, and let $z'_i$ be the first point in the walk outside the set $A_m(i-1)$---then $A_m(i):=A_m(i-1)\cup\{z'_i\}$.
	
	Importantly, the law of $A_m(i)$ does not depend on the order of $\{z_{m,1},...,z_{m,i}\}$, as proven by Diaconis and Fulton \cite{stanford1991growth}.
\end{definition}

We know from Levine and Peres \cite{Levine_2010} that the sets $A_m(m^2s)$ approach their deterministic limits $D_s$ almost surely. That is, for any $\eps>0$, we know that
\[d_H(\partial A_m(m^2s),\partial D_s)\leq\eps\]
almost surely for sufficiently large $m$, where the boundary $\partial A_m(m^2s)$ denotes the points in $A_m(m^2s)$ adjacent to $A_m(m^2s)^c$. Here and below, we use $d_H$ to denote the Hausdorff distance between sets:
\[d_H(A,B):=\inf\nolimits_{x\in A}\sup\nolimits_{y\in B}d(x,y).\]
In the following sections, we will use this result along with an iterative argument to recover a stronger convergence rate on $A_m(m^2s)$.
	\section{Main result}
We will write $(A)^\eps$ and $(A)_\eps$ for the outer- and inner-$\eps$-neighborhoods of a set $A\subset\mathbb{R}^2$, respectively. That is,
\[(A)^\eps := \{z\in\mathbb{R}^2\;|\;d(z,A)<\eps\},\qquad (A)_\eps := \{z\in A\;|\;d(z,A^c)>\eps\}.\]
Our primary result is the following convergence rate on the IDLA occupied sets $A_m(t)$ to their deterministic scaling limits $D_s$:
\begin{theorem}\label{biggun'}
	Suppose $D_\tau$ is a smooth flow arising from a concentrated mass distribution. For large enough $m$, the fluctuation of the associated IDLA $A_m(t)$ is bounded as
	\[
	\mathbb{P}\bigg\{(D_{s})_{C_5m^{-3/5}}\cap\frac{1}{m}\mathbb{Z}^2\subset A_m(m^2s)\subset (D_{s})^{C_5m^{-3/5}}\;\text{for all}\;s\in[0,T]\bigg\}^c\leq e^{-m^{2/5}}\]
	for a constant $C_5$ depending on the flow. Equivalently,
	\[\mathbb{P}\left[d_H(\partial A_m(m^2s),\partial D_s)>C_5m^{-3/5}\;\text{for any}\;s\in[0,T]\right]\leq e^{-m^{2/5}},\]
	where $d_H$ is the Hausdorff distance.
\end{theorem}
As mentioned in the introduction, we have reason to believe that the $m^{-3/5}$ convergence rate so described is non-optimal. Indeed, we will see in Lemma \ref{harmoniclemma2-2} that this results from a relatively rudimentary $L^1$ bound on the convergence rate of discrete Green's functions, rather than from the geometry of IDLA itself. We will discuss suggestions for further research in Section \ref{conclusion}.

\subsection{Overview of notation}
Henceforth, we will assume we have fixed a smooth, concentrated mass distribution, and we will use the language of Section \ref{background} to refer to it. That is, $T$ will always refer to the total volume of our source sets, $z_{m,i}$ to the $i^{th}$ source point in the resolution-$m$ discretization of our mass distribution, and $D_s$ to the scaling limit of IDLA started on the density $\sigma_s$.

To discuss fluctuations of IDLA away from its scaling limit, we define the following notions of ``earliness'' and ``lateness'':
\begin{itemize}
	\item We say that $z\in \frac{1}{m}\mathbb{Z}^2$ is $\eps$-early if $z\in A_m(\tau m^2)$ but $z\notin (D_\tau)^\eps$, for some $\tau>0$. Let $\Eps_\eps[t]$ be the event that some point in $A_m(t)$ is $\eps$-early.
	\item Similarly, we say that $z$ is $\eps$-late if $z\in (D_\tau)_\eps$ but $z\notin A_m(\tau m^2)$, for some $\tau>0$. Let $\mathcal{L}_\eps[t]$ be the event that some point in $(D_{t/m^2})_\eps$ is $\eps$-late.
\end{itemize}

As introduced in the preceding subsection, we will write $(A)^\eps$ and $(A)_\eps$ for the outer- and inner-$\eps$-neighborhoods of a set $A\subset\mathbb{R}^2$:
\[(A)^\eps := \{z\in\mathbb{R}^2\;|\;d(z,A)<\eps\},\qquad (A)_\eps := \{z\in A\;|\;d(z,A^c)>\eps\}.\]
Finally, for convenience and visual clarity, we will use $m^2s$ [resp., $m^2T$, etc.] in place of $\lfloor m^2s\rfloor$ in places where the meaning is clear. In particular, $A_m(sm^2):=A_m(\lfloor sm^2\rfloor)$.

\subsection{Required lemmas}
A number of existing results are necessary in the proof of Theorem \ref{biggun'}; we collect many of them here.

Firstly, we use the following two estimates on IDLA. The first bounds the probability of so-called ``thin tentacles''---shown in Figure \ref{tentaclefig}---and is simply a transcription of Lemma 2 of \cite{10.2307/23072157} in our setting. The second is a part of the estimate of Levine and Peres \cite{Levine_2010} earlier described, demonstrating that extremely late points are unlikely.
\begin{lemma}[Thin Tentacles]\label{tentacles}
	There are positive absolute constants $b$, $C_0$, and $c_0$ such that for all $z\in\frac{1}{m}\mathbb{Z}^2$ with $d(z,D_{0})\geq r$,
	\[\mathbb{P}[z\in A_m(t), \#(A_m(t)\cap B(z,r))\leq bm^2r^2]\leq C_0e^{-c_0mr}.\]
\end{lemma}
\begin{proof}
	The proof can be taken verbatim from Jerison et al.~\cite{10.2307/23072157}, with our scaling in mind.
\end{proof}
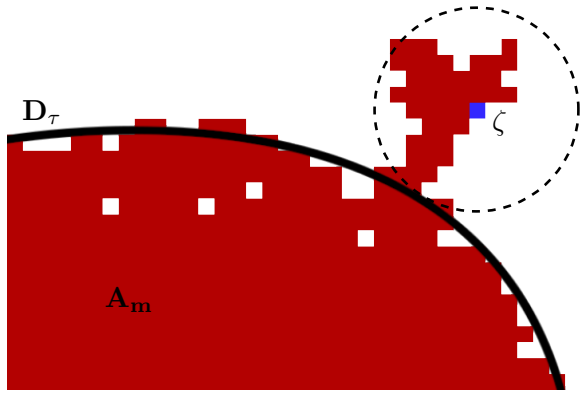
\begin{figure}[H]
	\centering
	\begin{tikzpicture}
		\node[anchor=south west,inner sep=0] (image) at (0,0) {\includegraphics[scale=.3]{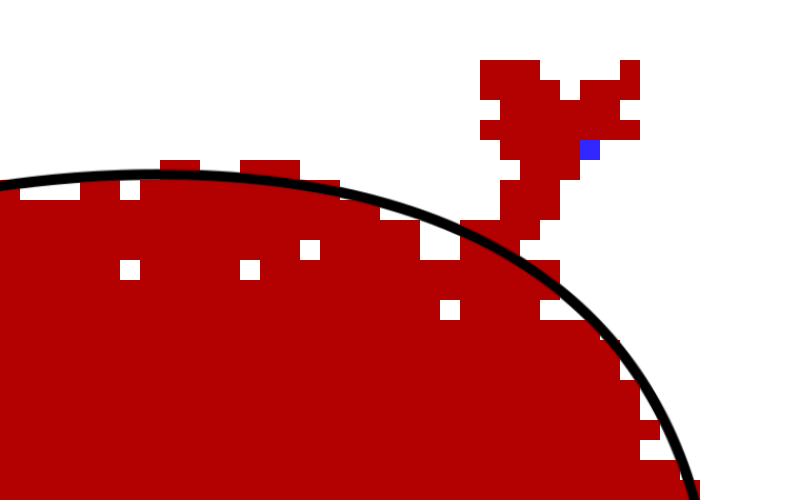}};
		\node[text width=2cm,rotate=0] at (2.3,1.2) {$\mathbf{A_m}$};
		\node[text width=2cm,rotate=0] at (1.2,3.7) {$\mathbf{D_\tau}$};
		\node[text width=2cm,rotate=0] at (7.45,3.55) {$\zeta$};
		\draw [dashed, line width=1pt] (6.24,3.73) circle(1.355cm);
	\end{tikzpicture}
	
	\caption{It is conceivable that the IDLA set $A_m$ extends out from its limit $D_s$ in thin tentacles, as pictured; we quantify this ``thinness'' near a point $\zeta$ by the number of filled spaces in disks centered at $\zeta$. In Lemma \ref{tentacles}, we show that a fixed positive fraction of each such disk (for small enough radii) is very likely filled in.}\label{tentaclefig}
\end{figure}
\begin{lemma}\label{levinepereslemma}
	There are absolute constants $C_0,c_0>0$ such that for all real $\eps>0$, $T\geq\tau\geq0$, and large enough $m$,
	\[\mathbb{P}\left(\mathcal{L}_\eps[\tau m^2]\right)\leq C_0e^{-c_0m^2/\log m}.\]
\end{lemma}
\begin{proof}
	By Levine and Peres \cite[p.~49]{Levine_2010}, the probability $\mathbb{P}_z$ that some $z\in\frac{1}{m}\mathbb{Z}^2\cap(D_\tau)_\eps$ is $\eps$-late is bounded as
	\[\mathbb{P}_z:=\mathbb{P}\left[z\in\frac{1}{m}\mathbb{Z}^2\cap(D_\tau)_\eps\;\text{is}\;\eps\text{-late}\right]\leq 4e^{-cm^2/\log m}\]
	for large enough $m$, where $c$ depends only on $S_m$, $D_\tau$, and $\eps$. Now, $\#\{z\in\frac{1}{m}\mathbb{Z}^2\cap(D_\tau)_\eps\}=O(m^2)$, so we can bound the total probability of $\mathcal{L}_\eps[\tau m^2]$ as
	\[\mathbb{P}\left(\mathcal{L}_\eps[\tau m^2]\right)\leq\sum_{z\in\frac{1}{m}\mathbb{Z}^2\cap(D_\tau)_\eps}\mathbb{P}_z\leq C_0m^2e^{-cm^2/\log m}\]
	for some $C_0$. Choosing some $c_0<c$, we have $m^2e^{-cm^2/\log m}\leq e^{-c_0m^2/\log m}$ for large enough $m$, and the lemma follows.
\end{proof}

The next two lemmas control the flow $s\mapsto D_s$. In short, the first shows that the arclength of $\partial D_s$ is uniformly bounded on both sides, and the second shows that $D_s$ grows at a linear rate at all points. The first follows directly from the smoothness of $D_s$.
\begin{lemma}\label{heleshawlength}
	For $s\in[0,T]$, the arclength of $\partial D_{s}$ is bounded as
	\[u\leq \op{Len}(\partial D_{s})\leq U,\]
	where $u,U>0$ are constants depending only on the flow.
\end{lemma}
\begin{lemma}\label{heleshawdist}
	For a smooth flow $D_s$ and any times $T\geq s_1\geq s_0\geq 0$,
	\[v(\sqrt{1+s_1}-\sqrt{1+s_0})\leq d(D_{s_1}^c, D_{s_0}):=\inf\{d(x,y)\;|\;x\in D_{s_1}^c,y\in D_{s_0}\}\]
	and
	\[V(\sqrt{1+s_1}-\sqrt{1+s_0})\geq d_H( D_{s_1},D_{s_0}):=\inf\nolimits_{x\in D_{s_1}}\sup\nolimits_{y\in D_{s_0}}d(x,y)\]
	where $v,V>0$ are constants depending only on $S_m$.
\end{lemma}
\begin{proof}
	The upper bound follows from the smoothness of $D_s$ and the compactness of the interval $[0,T]$.
	
	For the lower bound, we will exploit the fact that $D_s$ is also the scaling limit of divisible sandpile processes on $\{z_{m,1},...,z_{m,n_m}\}$ with starting set $D_0$. We will not give details on the divisible sandpile process here; see \cite{Levine_2010} for more details on scaling limits of divisible sandpiles.
	
	Choose an $s\in[0,T]$, and let $D^{1/m}_s(t)$ be the fully occupied set of the divisible sandpile on the lattice $\frac{1}{m}\mathbb{Z}^2$ with starting density
	\[\mathbf{1}_{D_s}+\sum_{i=m^2s}^{m^2s+t}\mathbf{1}_{z_{m,i}}.\] 
	In the interval $[s,s+\eps]$, a total of $\eps m^2$ particles are released---in fact, one particle is started at $z_{m,n}$ at each time $n/m^2$. From Lemmas \ref{harmoniclemma2-1}(d) and \ref{harmoniclemma2-2}(a,b), we can bound the exit probability as
	\[\mathbb{P}\left(z_{m,n}\;\text{exits}\;\frac{1}{m}\mathbb{Z}\cap D_{s}\;\text{at}\;z'\in\partial D_{s}\right)\geq\frac{c}{m},\]
	which tells us that, in the divisible sandpile model, we need $m/c$ particles to ensure that the new set $D^{1/m}_s(m/c)$ contains the $m^{-1}$-ball around $\frac{1}{m}\mathbb{Z}^2\cap D_s$: 
	\[D^{1/m}_s(m/c)\supset \left(\frac{1}{m}\mathbb{Z}^2\cap D_s\right)^{1/m}.\]
	Now, we can apply the same estimate to the expanded set
	\[\frac{1}{m}\mathbb{Z}^2\cap(D_s)^{1/m}\subset D^{1/m}_s(m/c).\]
	That is, if $z'$ is in the boundary of both $\frac{1}{m}\mathbb{Z}^2\cap(D_s)^{1/m}$ and $D^{1/m}_s(m/c)$, we have
	\[\mathbb{P}\left(z_{m,n}\;\text{exits}\;D^{1/m}_s(m/c)\;\text{at}\;z'\right)\geq\mathbb{P}\left(z_{m,n}\;\text{exits}\;\frac{1}{m}\mathbb{Z}\cap (D_s)^{1/m}\;\text{at}\;z'\in\partial D_{s}\right)\geq\frac{c}{m},\]
	and thus 
	\[D^{1/m}_s(2m/c)\supset \left(\left(\frac{1}{m}\mathbb{Z}^2\cap D_s\right)^{1/m}\right)^{1/m}\supset \left(\frac{1}{m}\mathbb{Z}^2\cap D_s\right)^{\sqrt{2}/m}.\]
	Continuing in this manner, we find that
	\[D^{1/m}_s(hm/c)\not\subset\left(\frac{1}{m}\mathbb{Z}^2\cap D_s\right)^{h/\sqrt{2}m},\]
	and in particular that
	\[D^{1/m}_s(\eps m^2)\not\subset\left(\frac{1}{m}\mathbb{Z}^2\cap D_s\right)^{\eps c/\sqrt{2}}.\]
	Now, $D_{s+\eps}$ is the scaling limit of these sets $D^{1/m}_s(\eps m^2)$, so we find that $d(D_{s+\eps}^c, D_{s})\geq \eps c/\sqrt{2}$. Then $\partial_\eps d(D_{s+\eps}^c, D_{s})\geq c/\sqrt{2}$ for all $\eps>0$, which implies the claim.
\end{proof}

Finally, the following two lemmas control the exit times of Brownian motion from an interval $[a,b]$. These are restatements of Lemmas 5 and 6 in \cite{10.2307/23072157}, so we omit the proofs here. Below, let $B(s)$ be centered, one-dimensional Brownian motion, and denote
\[\tau(-a,b)=\inf\{s>0\;|\;B(s)\notin[-a,b]\}.\]
\begin{lemma}\label{escapelemma}
	Let $0<a\leq b$. If $a+b\leq 3$, then
	\[\mathbb{E}e^{\tau(-a,b)}\leq 1+10ab.\]
\end{lemma}
\begin{lemma}\label{escapelemma2}
	For any $k,s>0$,
	\[\mathbb{P}\left\{\sup_{s'\in[0,s]}B(s')\geq ks\right\}\leq e^{-k^2s/2}.\]
\end{lemma}

\subsection{The recurrent potential kernel}\label{kernelsection}
Key to much of our analysis will be the so-called \emph{recurrent potential kernel} $g:\mathbb{Z}^2\to\mathbb{R}$, which acts as a free Green's function for the discrete Poisson equation. We define it in probabilistic terms as
\[g(z):=\sum_{n=0}^\infty(P_n(0)-P_n(z)),\]
where $P_n(z)$ is the probability that an $n$-step simple random walk from the origin in $\mathbb{Z}^2$ ends at $z$. Importantly,
\[\Delta_hg(x):=\tfrac{1}{4}(g(x+1)+g(x-1)+g(x+i)+g(x-i))-g(x)=\delta_{x,0}.\]
That is, $\Delta_hg(0)=1$, but $\Delta_hg(x\neq 0)=0$. We will also use the first few terms of the asymptotic expansion of $g$:
\[\left|g(z)-\lambda - \frac{2}{\pi}\log|z|\right|\leq\frac{C_1}{|z|^2}.\]
A complete expansion was discovered by Kozma and Schreiber \cite{kozma2004}, but we will not use it here.

We also consider discrete derivatives of $g$. Without loss of generality, choose a unit vector $\unit{n}=\alpha_1\unit{x}+\alpha_2(\unit{x}+\unit{y})$ in the ``east-northeast'' half-quadrant---i.e., with $1\geq\alpha_1,\alpha_2\geq0$.
Then define
\begin{equation}\label{discderiv}
	\partial_{\unit{n}}g:=\alpha_1g(z-1)+\alpha_2g(z-(1+i))-(\alpha_1+\alpha_2)g(z),
\end{equation}
which is discrete harmonic away from $\{0,1,1+i\}$. Now, extend both $g$ and $\partial_\unit{n}g$ by linear interpolation to the grid 
\[\mathcal{G}:=\{(x,y)\in\mathbb{R}^2\;|\;x\in\mathbb{Z}\;\text{or}\;y\in\mathbb{Z}\}.\]

Choose a constant $c>0$ such that $\partial_{n}g(z)>0$ on the half-plane $\{z\in\mathcal{G}\;|\;z\cdot\unit{n}\leq c\}$; since the arc $[0,\pi/4]$ is compact, we can assume without loss of generality that $c$ holds this property for all $\unit{n}$. Numerical calculations show that we can take $c\gtrsim 1/5$.

For an integer $m\geq 1$, let $B_{R_0}^-=B_{mR_0}(mR_0\unit{n})$ be the radius $mR_0$ disk tangent to the origin in the direction $\unit{n}$. By Lemma 8(a) of \cite{10.2307/23072157}, we know that
\[\{z\in\mathcal{G}\;|\;\partial_{\unit{n}}g<-1/2mR'_0\}\subset (B_{R_0}^-)^{C_2}.\]
By the above discussion, this means that
\begin{equation}\label{fakeinclusion}
	\{z\in\mathcal{G}\;|\;\partial_{\unit{n}}g<-1/2mR_0\}\subset \{z\in\mathcal{G}\;|\;z\cdot\unit{n}> c\}\cap (B_{R_0}^-)^{C_2}\subset B^-_{R'_0},
\end{equation}
for any $R'_0\geq4C_2R_0/c$.
	\section{Early points imply late points}
The following sections make up the proof of Theorem \ref{biggun'}, split into three parts. First, we will show that the existence of an early point at time $t$ implies that of a similarly late point by the same time. For this, we use a harmonic function $H_\zeta(z)$ that has a pole at the proposed early point, $\zeta\in\frac{1}{m}\mathbb{Z}^2$, and we define a martingale $M_\zeta(t)$ (roughly) by summing the values of $H_\zeta(z)-H_\zeta(z_{m,i})$ over $A_m(t)$. Since $H_\zeta(\zeta)$ is large, the martingale takes a much larger value than expected at time $t$; we finish up by using Lemma \ref{escapelemma2} to show that this is unlikely.

In the following two sections, we set up the theory necessary for this first proof. 

\subsection{The discrete harmonic function $H_\zeta(z)$}
Choose $\zeta\in\frac{1}{m}\mathbb{Z}^2\cap(D_T\setminus D_{0})$, and let $\tau>0$ be such that $\zeta\in\partial D_\tau$. This is possible because the sets $\partial D_s$ for $s>0$ form a foliation of $D_T\setminus D_{0}$.

Without loss of generality, suppose the outward normal vector $\unit{n}$ to $\partial D_\tau$ at $\zeta$ is pointing into the ``east-northeast'' half-quadrant, or equivalently that $\unit{n}\cdot\unit{x}\geq\unit{n}\cdot\unit{y}\geq 0$. This subsumes other cases by reflecting the plane appropriately.

Now, write $\unit{n}=\alpha_1+\alpha_2(1+i)$. Because of the direction of $\unit{n}$, both $\alpha_1$ and $\alpha_2$ are positive and bounded below 1.

Define
\[H_\zeta(z)=\frac{\pi}{2}[\alpha_1g(mz-m\zeta-1)+\alpha_2g(mz-m\zeta-(1+i))-(\alpha_1+\alpha_2)g(mz-m\zeta)].\]
We can view this as a directional derivative of the potential kernel in the direction opposite $\unit{n}$. We will extend this by linear interpolation to the grid $\mathcal{G}_m=\{(x,y)\in\frac{1}{m}\mathbb{R}^2\;|\;x\in\frac{1}{m}\mathbb{Z}\;\text{or}\;y\in\frac{1}{m}\mathbb{Z}\}$.

This function is designed to be a discrete-harmonic approximation of the continuum function
\[F_\zeta(z)=\op{Re}\left(\frac{\unit{n}/m}{\zeta-z}\right),\]
pictured in Figure \ref{plotfig}, where we view $\unit{n}$ as a complex number.
\begin{figure}[H]
	\centering
	\includegraphics[scale=.5]{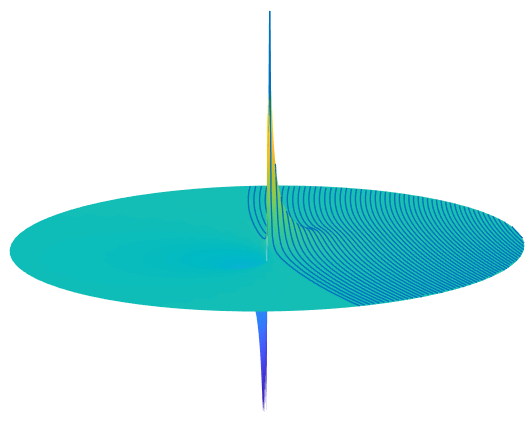}
	\caption{A plot of $F_\zeta$, with an example domain $D_\tau$ marked out by dark blue curves. Importantly, $F_\zeta$ has a large positive pole within $D_\tau$ at $\zeta$, but its negative pole lies entirely outside $D_\tau$. The discrete harmonic function $H_\zeta$ closely approximates this function away from the pole $\zeta$. }\label{plotfig}
\end{figure}

Now, choose $R'_0=R'_0(\tau)$ such that the two disks $B^+$ and $B^-$ of radius $R'_0$ tangent to $\partial D_\tau$ at any point lie entirely inside and outside $D_\tau$, respectively. Note that $R'_0$ is bounded away from zero, as $[0,T]$ is compact and $R'_0> 0$ for all time. Let $R_0=cR'_0/4C_2$, as in Equation \ref{fakeinclusion}, and define the following subsets of $\frac{1}{m}\mathbb{Z}^2$:
\[\Omega^1_\zeta=\mathcal{G}_m\cap D_\tau,\]
\[\Omega^2_\zeta = \{z\in\mathcal{G}_m\;|\;H_\zeta(z)>1/2mR_0\}\setminus\{\zeta+u\;|\;u\in(0,m^{-1})\}.\]
In short, $\Omega^1_\zeta$ is the discretized version of the Hele--Shaw level set $D_\tau$, and $\Omega^2_\zeta$ is an approximation of the ``inner'' radius $R_0$ circle tangent to $\partial D_\tau$ at $\zeta$. We will combine these as
\[\Omega_\zeta=\Omega^1_\zeta\cup\Omega^2_\zeta.\]
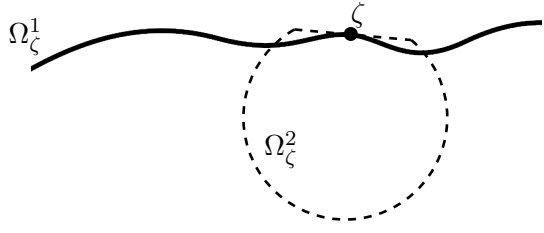
\begin{figure}[H]
	\centering
	\begin{tikzpicture}
		\node[anchor=south west,inner sep=0] (image) at (0,0) {\includegraphics[scale=1]{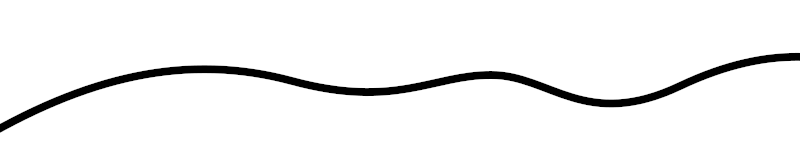}};
		\node[text width=2cm,rotate=0] at (.7,.7) {$\Omega_\zeta^1$};
		\node[text width=2cm,rotate=0] at (4.1,-.8) {$\Omega_\zeta^2$};
		\node[text width=2cm,rotate=0] at (5.25,1) {$\zeta$};
		\draw [line width=3pt] (4.25,.74) circle(.04cm);
		\draw [dashed, line width=1pt] (3.5,0.8) arc(120:410:1.355cm);
		\draw [dashed, line width=1pt] (3.5,0.8)--(5.05,0.66);
	\end{tikzpicture}
	\caption{We form our domain $\Omega_\zeta$ by combining the subsets $\Omega^1_\zeta$ and $\Omega_\zeta^2$; the latter guarantees that $H_\zeta$ is not too large on the boundary, but it may also affect the regularity of the boundary.}\label{domainfig}
\end{figure}
An example is pictured in Figure \ref{domainfig}. We summarize many of the basic properties of $H_\zeta$ and $\Omega_\zeta$ in the following lemma:
\begin{lemma}\label{harmoniclemma}
	For any $m$, $H_\zeta$ and $\Omega_\zeta$ satisfy the following properties:
	\begin{enumerate}[label=\emph{(\alph*)}]
		\item $H_\zeta$ is grid-harmonic in the interior of $\Omega_\zeta$, and $H_\zeta\geq-\frac{1}{2mR_0}$.
		\item $\zeta\in\partial\Omega_\zeta$, and for all $z\in\partial\Omega_\zeta\setminus\{\zeta\}$, we have $|H_\zeta(z)|\leq\frac{1}{2mR_0}$.
		\item There is an absolute constant $C_1<\infty$ such that
		\[|H_\zeta(z)-F_\zeta(z)|\leq C_1m^{-2}|z-\zeta|^{-2}.\]
		In particular, if $R_1=\inf_{z_{m,i}}d(z_{m,i},\zeta)$, then
		\[|H_\zeta(z_{m,i})-F_\zeta(z_{m,i})|\leq C_1m^{-2}R_1^{-2}.\]
		\item $1\leq H_\zeta(\zeta)\leq 2$.
	\end{enumerate}
\end{lemma}
\begin{proof}$ $\medskip
	
	\noindent(a) By definition, $H_\zeta$ is grid-harmonic everywhere except for $\zeta$, $\zeta+1/m$, and $\zeta+(1+i)/m$. Firstly, $\zeta$ itself lies on the boundary of $\Omega_\zeta$ by definition. As the normal vector $\unit{n}$ to $\partial D_\tau$ at $\zeta$ points into the east-northeast half-quadrant, for large enough $m$, neither of the remaining points can lie in $D_\tau$ (and thus in $\Omega^1_\zeta$). Furthermore, $H_\zeta$ is negative at both points, as in \cite{10.2307/23072157}, so they cannot lie in $\Omega_\zeta^2$. Thus, they cannot lie in $\Omega_\zeta$, so $H_\zeta$ is grid-harmonic in that set.
	
	The lower bound follows from Equation \ref{fakeinclusion}.
	
	\noindent(b) As in part (a), the lower bound $H_\zeta(z)\geq-1/2mR_0$ is clear from Equation \ref{fakeinclusion}. The upper bound follows from the inclusion of $\Omega_\zeta^2$, as the boundary of $\Omega_\zeta$ must lie at or outside the boundary of $\Omega_\zeta^2$.
	
	\noindent(c, d) The last points are exactly Lemma 7(c, d) in \cite{10.2307/23072157}, as our notions of $H_\zeta$ and $F_\zeta$ are simply rotations of theirs.
\end{proof}

\begin{lemma}\label{harmoniclemma2}
	\begin{enumerate}[label=\emph{(\alph*)}]
		\item[]
		\item There is an absolute constant $C_2<\infty$ such that 
		\[\frac{1}{m}\mathbb{Z}^2\cap D_{\tau}\subset \Omega_\zeta\subset \frac{1}{m}\mathbb{Z}^2\cap(D_{\tau})^{C_2/m}.\]
		\item For any $U\subset\Omega_\zeta$, then 
		\[-\frac{1}{2mR_0}\leq H_\zeta(z)\leq\frac{1}{md(\zeta,U)-C_2}\]
		whenever $md(\zeta,U)>C_2$ and $z\in U$.
		\item For all $0<s<\tau$,
		\[\left|\sum_{z\in D_s\cap\frac{1}{m}\mathbb{Z}^2}H_\zeta(z)-\sum_{i=0}^{m^2s}H_\zeta(z_i)\right|\leq C_2\log m.\]
	\end{enumerate}
\end{lemma}

\begin{proof}$ $\medskip
	
	\noindent(a) As shown in \cite{10.2307/23072157}, the level sets of $H_\zeta$ differ from the level curves of $F_\zeta$ by at most a fixed distance $C_2/m=2C_1/m$. In particular, $\Omega_\zeta^2\subset (B^+)^{C_2/m}$, where $B^+$ is the disk of radius $R_0$ contained within $D_\tau$ and tangent to $\partial D_\tau$ at $\zeta$.
	
	By construction, $(B^+)^{C_2/m}\subset (D_\tau)^{C_2/m}$. Thus, by adding $\Omega_\zeta^2$, we never modify points in $\Omega_\zeta^1$ outside the narrow strip $(D_\tau)^{C_2/m}\setminus D_\tau$.
	
	\noindent(b) The proof of this fact is the same as that of Lemma 8(b) of \cite{10.2307/23072157}, but now using the fact that $\sup_U F_\zeta\leq\frac{1}{md(\zeta,U)}$.
	
	\noindent(c) Let $s_0$ be maximal such that $D_{s_0}\subset (D_\tau)_{4C_1/m}$---by Lemma \ref{heleshawdist}, we know that
	\begin{equation}\label{hausdorffdisteq1}
		d_H(\partial D_{\tau},\partial D_{s_0})\leq \frac{V}{v}d(\partial D_{\tau},\partial D_{s_0})\leq 4VC_1/mv.
	\end{equation}
	Write $D_s\cap\frac{1}{m}\mathbb{Z}^2=A\cup B$, where 
	\[A=\left(D_{s}\cap\frac{1}{m}\mathbb{Z}^2\right)\cap D_{s_0}=D_{s\wedge s_0}\cap\frac{1}{m}\mathbb{Z}^2,\]
	\[B=\left(D_s\cap\frac{1}{m}\mathbb{Z}^2\right)\setminus D_{s_0}.\]
	Now, choose $R>0$ with the following properties:
	\begin{itemize}
		\item The union $U_{R}$ of the disks of radius $R$ centered at $\zeta \pm (R+2C_1/m)$ and at $\zeta\pm i(R+2C_1/m)$ is connected.
		\item The connected component of $\zeta$ in $\mathbb{R}^2\setminus U_{R}$ is contained within $B_{4C_1/m}(0)$
		\item We have $B_{R}(\zeta)\setminus B_{4C_1/m}(\zeta)\subset U_{R}$.
	\end{itemize}
	\begin{figure}[H]
		\centering
		\begin{tikzpicture}
			\clip (-3,-.3) rectangle (8,4.3);
			\node[text width=2cm,rotate=0] at (4.15,2.1) {$\zeta$};
			\draw [line width=3pt] (3,2) circle(.04cm);
			\draw [dashed, line width=1pt,fill=yellow, opacity=0.2] (3,-.4) circle(2cm);
			\draw [dashed, line width=1pt,fill=yellow, opacity=0.2] (5.4,2) circle(2cm);
			\draw [dashed, line width=1pt,fill=yellow, opacity=0.2] (.6,2) circle(2cm);
			\draw [dashed, line width=1pt,fill=yellow, opacity=0.2] (3,4.4) circle(2cm);
		\end{tikzpicture}
		\caption{An illustration of $U_R$ (yellow). Note that the disks of radius $R$ do not get ``too close'' to $\zeta$, so the sum $\sum_{z\in U_R}\frac{1}{m^2|z-\zeta|^2}$ is of order $\log m$.}\label{circfig}
	\end{figure}
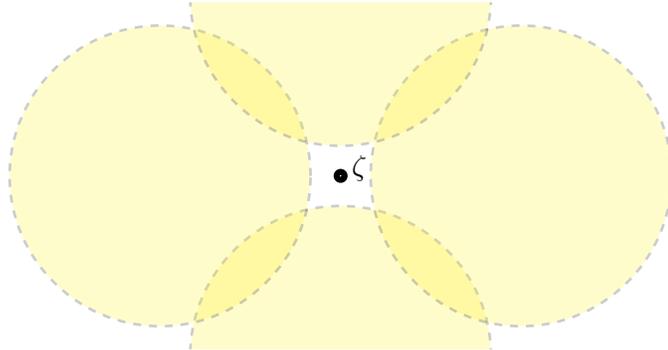
	Note that, for any $r>R$, this implies $B_{r}(0)\setminus B_{4C_1/m}(0)\subset U_{r}$. Define 
	\[R'=\max\left(R,\op{diam}(D_0)+2V\sqrt{1+\tau}\right).\]
	Importantly, by Lemma \ref{heleshawdist},
	\[D_\tau\subset B_{R'}(\zeta).\]
	Since $D_{s_0}$ does not intersect $B_{4C_1/m}$, this means that $A\subset D_{s_0}\subset U_{R'}$. By Lemma \ref{harmoniclemma}(c),
	\begin{align*}
		\left|\sum_{z\in A}\left(H_\zeta(z)-F_\zeta(z)\right)-\sum_{i=1}^{m^2(s\wedge s_0)}\left(H_\zeta(z_i)-F_\zeta(z_i)\right)\right|&\leq\sum_{z\in A}\frac{C_1}{m^2|z-\zeta|^2}+\frac{s_0}{|\zeta|^2}\\
		&\leq\sum_{z\in U_{R'}\cap\frac{1}{m}\mathbb{Z}^2}\frac{C_1}{m^2|z-\zeta|^2}+\frac{s_0}{|\zeta|^2}\\
		&\leq 32C_1\log mR'+\frac{s_0}{|\zeta|^2}\\
		&\leq \frac{C_2}{6}\log m,
	\end{align*}
	for an appropriately chosen $C_2$. Using the bound $|\nabla F_\zeta(z)|\leq C|z-\zeta|^{-2}$, we bound
	\begin{align*}
		\left|m^2\int_{D_{s\wedge s_0}}F_\zeta-\sum\nolimits_{A}F_\zeta\right|\leq C\int_{D_{s\wedge s_0}}\frac{dA}{|z-\zeta|^2}\leq C\int_{U_{R'}}\frac{dA}{|z-\zeta|^2}\leq\frac{C_2'}{6}\log m,
	\end{align*}
	and similarly,
	\begin{align*}
		\left|m^2\int F_\zeta\sigma_{s\wedge s_0}-\sum_{i=1}^{m^2(s\wedge s_0)}F_\zeta(z_i)\right|\leq\frac{C_2'}{6}\log m.
	\end{align*}
	By Lemma \ref{quaddomain}, however, 
	\[\int_{D_{s\wedge s_0}}F_\zeta (1-\sigma_{s\wedge s_0})=0,\]
	and we thus find
	\[\left|\sum_{z\in A}H_\zeta(z)-\sum_{i=1}^{m^2(s\wedge s_0)}H_\zeta(z_i)\right|\leq\frac{C_2'}{2}\log m.\]
	Finally, we must show that the contribution of $B$ (if it is nonempty) to the sum is negligible.
	From Equation \ref{hausdorffdisteq1}, we know that $D_s\setminus D_{s_0}\subset D_\tau\setminus(D_\tau)_{4VC_1/mv}$ and thus that $B\subset D_\tau\setminus(D_\tau)_{4VC_1/mv}$. Thus, there are at most $Um\cdot4VC_1/mv$ points in $B$; since $H_\zeta$ decreases as $1/md(z,\zeta)$ around the edge of $D_\tau$, we find that
	\[\left|\sum_{z\in B}H_\zeta(z)\right|\leq \frac{4UVC_1}{v}\cdot C\log m\leq\frac{C'}{4}\log m.\]
	Similarly, there are at most $4VUC_1/v$ source points between times $s_0$ and $s$, so we bound the final term as
	\[\left|\sum_{i=m^2s_0}^{ m^2s}H_\zeta(z_{m,i})\right|\leq\frac{4VUC_1}{v}\cdot\sup_{z_{m,i}}\left|H_\zeta(z_{m,i})\right|=O(m^{-1})\leq \frac{C'}{4}\log m.\]
	Putting these contributions together implies the lemma.
\end{proof}

	\subsection{The martingale $M_\zeta(t)$}\label{martingalesection}
The harmonic function $H_\zeta$ gives rise to a natural martingale associated to our IDLA process, using the concept of a \emph{grid Brownian motion}:
\begin{definition}
	A \emph{grid Brownian motion} starting at the point $x\in\frac{1}{m}\mathbb{Z}^2$ is a random process $t\mapsto W_t\in\mathcal{G}_m$ defined as follows.
	
	Let $B_t$ be an origin-centered Brownian motion, and for each integer $n\geq 1$, let $\tau_n^*>0$ be the $n^{th}$ time that $B_t$ visits a point in $\frac{1}{m}\mathbb{Z}$. For each $n$, choose a uniform random direction $\unit{u}_n\in\{1,i\}$. For $t\in[\tau_n^*,\tau_{n+1}^*]$, define
	\[W_t=x+\sum_{j=1}^{n-1}\left(B_t(\tau_{j+1}^*)-B_t(\tau_j^*)\right)\unit{u}_j+\left(B_t(t)-B_t(\tau_n^*)\right)\unit{u}_n.\]
	In short, $W_t$ is simply the process $B_t$, but turning in a random direction at each lattice point.
\end{definition}
For $k\in\{0,1,2,...\}$, let $\tilde{\beta}_k(s)$ be independent Brownian motions on the grid $\mathcal{G}_m$, starting at the source points $z_{m,k}$. We will define a modified IDLA process $A_\zeta(t)$ by induction. Let $A_\zeta(0)=A_m(0)=\frac{1}{m}\mathbb{Z}^2\cap D_0$, and let
\[s_k^*=\inf\left\{s\geq 0\;|\;\tilde{\beta}_k(s)\in\big(\tfrac{1}{m}\mathbb{Z}^2\setminus A_\zeta(k)\big)\cup\big(\mathcal{G}_m\setminus\Omega_\zeta\big)\right\}.\]
Then set $\beta_k(s)=\tilde{\beta}_k\big(\min\big(\frac{s}{1-s},s_k^*\big)\big)$, and set $A_\zeta(k+s)=A_\zeta(k)\cup\{\beta_k(s)\}$ for $0\leq s\leq 1$.

Since $H_\zeta$ is grid-harmonic, the process
\[M_\zeta(t):=\sum_{\ell=0}^{\lfloor t\rfloor-1}\left(H_\zeta(\beta_\ell(1))-H_\zeta(z_{m,\ell})\right)+H_\zeta(\beta_\ell(t-\lfloor t\rfloor))-H_\zeta(z_{m,\lfloor t\rfloor})\]
is a continuous-time martingale adapted to $\mathscr{F}_t=\sigma\{A_\zeta(s)\;|\;0\leq s\leq t\}$. By the Dubins--Schwarz theorem \cite[Theorem V.1.6]{revuz1991continuous}, we can write $M_\zeta(t)=B_\zeta(S_\zeta(t))$, where $S_\zeta(t)=\langle M_\zeta,M_\zeta\rangle_t$ is the quadratic variation of $M_\zeta$ and $B_\zeta$ is a standard Brownian motion.

For each $k$, $S_\zeta(k)$ is a stopping time w.r.t.~the filtration $\{\mathscr{F}_{T_\zeta(s)}\}_{s\geq 0}$, where $T_\zeta(s)=\inf\{t\;|\;S_\zeta(t)>s\}$. Further, $B_\zeta(s)$ is adapted to this filtration. By the strong Markov property, the processes 
\[\tilde{B}^k_\zeta(u):=B_\zeta(S_\zeta(k)+u)-B_\zeta(S_\zeta(k))\]
are independent Brownian motions started at zero.

Finally, for $-a<0<b$, write $\tau_k(-a,b)=\inf\{u>0\;|\;\tilde{B}_\zeta^k(u)\notin[-a,b]\}$. We will use these exit times in accordance with the following lemma, which is just a restatement of Lemma 9 of \cite{10.2307/23072157} in our setting:
\begin{lemma}\label{escapetime}
	Fix $\zeta\in\frac{1}{m}\mathbb{Z}^2\setminus D_{0}$, and let
	\[-a_k=\min_{z\in\partial A_\zeta(k)}(H_\zeta(z)-H_\zeta(z_{m,k+1})),\qquad b_k=\max_{z\in\partial A_\zeta(k)}(H_\zeta(z)-H_\zeta(z_{m,k+1})).\]
	Then
	\[S_\zeta(k+1)-S_\zeta(k)\leq\tau_k(-a_k,b_k).\]
\end{lemma}

We now proceed with the technique mentioned at the beginning of this section. That is, we will use the martingales $M_\zeta$ to detect the presence of a late or early point at $\zeta$; if either is the case, then $M_\zeta$ will be either much larger or much smaller than its mean. In turn, Lemma \ref{escapelemma2} will imply that this scenario is unlikely for small times $S_\zeta$. With the following two lemmas, we will be able to show that $S_\zeta$ is small on the event $\Eps_{a/m}(t)^c$, allowing the above argument to go through.

\begin{lemma}\label{insideest}
	Suppose $D_s$ is a smooth flow arising from an initial mass distribution. For 
	\[m\geq \max(3a+C_2,2C_2/\inf\nolimits_\zeta R_1),\]
	all $s\in[0,T]$, and $\zeta\notin (D_{s})^{(4a+2C_2)/m}$, we have
	\[\mathbb{E}\left[e^{S_\zeta(m^2s)}\mathbf{1}_{\Eps_{(a+1)/m}(m^2s)^c}\right]\leq m^K,\]
	where $K$ is a constant depending only on the flow.
\end{lemma}
\begin{proof}
	On the event $\Eps_{(a+1)/m}[t]^c$, we have $A_m(n)\subset (D_{n/m^2})^{(a+1)/m}$ for all $n\leq m^2s$. Since $\zeta\notin (D_{s})^{(4a+2C_2)/m}$, Lemma \ref{heleshawdist} tells us that
	\begin{align*}
		d(\zeta,A_m(n))&\geq \frac{4a+2C_2}{m}+d(\partial D_{s},\partial D_{n/m^2})-\frac{a+1}{m}\\
		&\geq\frac{1}{m}\left(3a+2C_2-1\right)+v\big(\sqrt{1+s}-\sqrt{1+n/m^2}\big),
	\end{align*}
	and thus (using Lemmas \ref{harmoniclemma}(b) and \ref{harmoniclemma2}(b)) that
	\[-\frac{1}{2mR_0}-\frac{1}{mR_1-C_2}\leq H_\zeta(z)-H_\zeta(z_{m,i})\]
	and that
	\[H_\zeta(z)-H_\zeta(z_{m,i})\leq \frac{1}{3a+C_2+mv\big(\sqrt{1+s}-\sqrt{1+n/m^2}\big)-1}+\frac{1}{2mR_0}.\]
	Now, choose $m$ large enough that $mR_1-C_2\geq mR_1/2$, and write $R_2=\min(R_0/2,R_1/4)$. 
	Then we know from Lemma \ref{escapetime} that
	\[(S_\zeta(n)-S_\zeta(n-1))\mathbf{1}_{\Eps_{(a+1)/m}[t]^c}\leq\tau_n(-c,b_n),\]
	with
	\[c=\frac{1}{mR_2},\qquad b_n=\frac{1}{3a+C_2+mv\big(\sqrt{1+s}-\sqrt{1+n/m^2}\big)-1}+\frac{1}{2mR_0}.\]
	Using Lemma \ref{escapelemma} along with the fact that $\tau_n(-c,b_n)$ are independent,
	\[\log\mathbb{E}\left[e^{S_\zeta(m^2s)}\mathbf{1}_{\Eps_{(a+1)/m}[m^2s]^c}\right]=\sum_{n=1}^{m^2s}\log\mathbb{E}e^{\tau_n(-c,b_n)}\leq\sum_{n=1}^{m^2s}\log(1+10cb_n)\leq\sum_{n=1}^{m^2s}10cb_n.\]
	Now, write $r_1=3a+2C_2+mv\sqrt{1+s}-1$, and calculate
	\begin{align*}
		\sum_{n=1}^{m^2s}b_n&\leq \int_0^t\frac{dn}{3a+C_2+mv\big(\sqrt{1+s}-\sqrt{1+n/m^2}\big)-1}+\frac{m^2s}{2mR_0}\\
		&= \int_0^{m^2s}\frac{dn}{r_1-v\sqrt{m^2+n}}+\frac{ms}{2R_0}\\
		&= 2v^{-1}\int_{vm}^{vm\sqrt{1+s}}\frac{xdx}{r_1-x}+\frac{ms}{2R_0}\\
		&= 2v^{-1}\int_{r_1-vm\sqrt{1+s}}^{r_1-mv}\frac{(r_1-y)dy}{y}+\frac{ms}{2R_0}\\
		&\leq 2v^{-1}r_1\log\left(\frac{r_1-mv}{r_1-vm\sqrt{1+s}}\right)+\frac{ms}{2R_0}\\
		& \leq 4v^{-1}m\sqrt{1+s}\log(vm\sqrt{1+s}/C_2)+\frac{s}{2R_0},
	\end{align*}
	so long as $mv\geq 3a+2C_2$. We thus find that
	\[\log\mathbb{E}\left[e^{S_\zeta(t)}\mathbf{1}_{\Eps_{(a+1)/m}[t]^c}\right]\leq \frac{40\sqrt{1+s}}{vR_2}\log(mv\sqrt{1+s}/C_2)+\frac{s}{2R_0R_2}.\]
	The theorem follows with $K>40\sqrt{1+T}/vR_2$.
\end{proof}

\begin{lemma}\label{outsideest}
	Suppose $D_s$ is smooth, and fix $a\geq 2C_2+2$, $\ell\leq a$, and $s\in[0,T]$. For 
	\[m\geq \max(3a+C_2,5a/\inf\nolimits_\zeta R_1)\]
	and $\zeta\in\frac{1}{m}\mathbb{Z}^2\cap \left((D_{s})_{\ell/m}\setminus D_{0}\right)$, we have
	\[\mathbb{E}\left[e^{S_\zeta(m^2s)}\mathbf{1}_{\Eps_{(a+1)/m}(m^2s)^c}\right]\leq m^Ke^{K'a},\]
	where $K$ is as in Lemma \ref{insideest} and $K'
	>0$ is another absolute constant.
\end{lemma}
\begin{proof}
	Recall that $R_1=\inf_{z_{m,i}}d(z_{m,i},\zeta)$. Since $mR_1\geq 5a$ and $a\geq 2C_2+2$, we can choose $t_0\geq 1$ such that $D_{t_0/m^2}\subset (D_{\tau})_{(4a+2C_2+1)/m}$. By Lemma \ref{heleshawdist}, we can take $t_0$ to satisfy
	\[m\sqrt{1+s}-\sqrt{m^2+t_0}\leq 2v^{-1}(4a+2C_2+1),\]
	and thus
	\[m^2s-t_0 \leq 2m\sqrt{1+T}(m\sqrt{1+s}-\sqrt{m^2+t_0})\leq 4v^{-1}m\sqrt{1+T}(4a+2C_2+1),\]
	Further, since $\zeta\notin (D_{t_0/m^2})^{(4a+2C_2)/m}$, Lemma \ref{insideest} gives
	\[\mathbb{E}\left[e^{S_\zeta(t_0)}\mathbf{1}_{\Eps_{a/m}(t_0)^c}\right]\leq\mathbb{E}\left[e^{S_\zeta(t_0)}\mathbf{1}_{\Eps_{(a+1)/m}(t_0)^c}\right]\leq m^K.\]
	As discussed in the proof of Lemma \ref{insideest}, for $m\geq 2C_2/\inf\nolimits_\zeta R_1$, we have
	\[-\frac{1}{mR_2}\leq H_\zeta(z)-H_\zeta(z_{m,i}),\]
	with $R_2=\min(R_0/2,R_1/4)$. We also know that $H_\zeta(z)-H_\zeta(z_{m,i})\leq 2+\frac{1}{mR_1-C_2}\leq 5/2$ from Lemma \ref{harmoniclemma2}(b), so we get
	\[S_\zeta(n)-S_\zeta(n-1)\leq\tau_n(-1/mR_2,5/2)\]
	and thus
	\begin{align*}
		\log\mathbb{E}\left[e^{S_\zeta(t)}\mathbf{1}_{\Eps_{(a+1)/m}(t)^c}\right]&=\log\mathbb{E}\left[e^{S_\zeta(t_0)}\mathbf{1}_{\Eps_{(a+1)/m}(t_0)^c}\right]\\
		&\qquad+\sum_{n=t_0+1}^{m^2s}\log\mathbb{E}\left[e^{S_\zeta(n)-S_\zeta(n-1)}\mathbf{1}_{\Eps_{(a+1)/m}(m^2s)^c}\right]\\
		&\leq K\log m+\sum_{n=t_0+1}^{m^2s}\log\mathbb{E}\left[e^{\tau_n(-1/mR_2,5/2)}\right]\\
		&\leq K\log m+(m^2s-t_0)\log(1+25/mR_2)\\
		&\leq K\log m+100v^{-1}\sqrt{1+T}(4a+2C_2+1)/R_2\\
		&\leq K\log m+100v^{-1}\sqrt{1+T}(4a+2C_2+1)/R_2.
	\end{align*}
	Now, $2C_2+1\leq a$ by hypothesis, so the claim follows.
\end{proof}
	\subsection{First estimate}\label{firstestsection}
Choose constants 
\[C_3=\max\left(\frac{24VC_2}{vb},\frac{72V}{vb},3/c_0\right),\]
\[\alpha = \frac{v^2b}{288UV^2K''}.\]

\begin{lemma}\label{firstest}
	For large enough $m$, $s\in[0,T]$, $3a+C_2\geq a\geq C_3m^{2/5}$, and $\ell\leq \alpha a$, we have
	\[\mathbb{P}(\Eps_{a/m}[m^2s]\cap\mathcal{L}_{\ell/m}[m^2s]^c)\leq e^{-2m^{2/5}}.\]
\end{lemma}
\begin{proof}[Step 1]\let\qed\relax
	For each integer $1\leq t\leq m^2s$ and each lattice point $z$, let
	\[Q_{z,t}=\{z\in A_m(t)\setminus A_m(t-1)\}\cap\{z\notin (D_{t/m^2})^{a/m}\}\cap\Eps_{a/m}[t-1]^c\]
	be the event wherein $z$ first joins the cluster at time $t$ and is the first $a/m$-early point. Now,
	\[\bigcup_{t\leq m^2s}\bigcup_{z\in (D_0)^{T/m}}Q_{z,t}=\Eps_{a/m}[m^2s].\]
	Fix $z\in \frac{1}{m}\mathbb{Z}^2\setminus (D_{t/m^2})^{a/m}$, and let $\zeta=\zeta(z,t)$ be the nearest point to $z$ in the annulus
	\begin{equation}\label{annulus}
		\frac{1}{m}\mathbb{Z}^2\cap(D_{t/m^2})^{V(4a+2C_2)/mv+2/m}\setminus(D_{t/m^2})^{V(4a+2C_2)/mv}.
	\end{equation}
	Since $\zeta\notin (D_{t/m^2})^{(4a+2C_2)/m}$, we have by Lemma \ref{insideest} that
	\[\mathbb{E}\left[e^{S_\zeta(t)}\mathbf{1}_{\Eps_{(a+1)/m}(t)^c}\right]\leq m^K.\]
	Let $M=6m^{2/5}$, so that Markov's inequality gives
	\[\mathbb{P}\left(\Eps_{(a+1)/m}(t)^c\cap\{S_\zeta(t)>M\}\right)\leq e^{-6m^{2/5}}\mathbb{E}\left[e^{S_\zeta(t)}\mathbf{1}_{\Eps_{(a+1)/m}(t)^c}\right]\leq e^{-3m^{2/5}}.\]
	Now, since $A_m(t-1)\subset (D_{t/m^2})^{a/m}$ and $z$ is adjacent to $A_m(t-1)$, we must have $z\in (D_{t/m^2})^{(a+1)/m}$. Thus, $Q_{z,t}\subset\Eps_{(a+1)/m}(t)^c$, so
	\[\mathbb{P}\left(Q_{z,t}\cap\{S_\zeta(t)>M\}\right)\leq e^{-3m^{2/5}}.\]
\end{proof}
\begin{proof}[Step 2]\let\qed\relax
	On the event $Q_{z,t}$, we know that 
	\begin{equation}\label{doesnthitboundaryeq}
		A_m(t)\subset (D_{t/m^2})^{(a+1)/m},
	\end{equation}
	as no points are $(m+1)$-early. However, we also know that 
	\[d_H(D_{t/m^2},D_{\tau})\geq d(D_{t/m^2},\zeta)\geq V(4a+2C_2)/mv,\]
	which implies by Lemma \ref{heleshawdist} that
	\[d(D_{t/m^2},D_{\tau}^c)\geq\frac{v}{V}d_H(D_{t/m^2},D_{\tau})\geq(4a+2C_2)/m.\]
	In turn, Equation \ref{doesnthitboundaryeq} implies that
	\[d(A_m(t),D_{\tau}^c)\geq d((D_{t/m^2})^{(a+1)/m},D_{\tau}^c)\geq d(D_{t/m^2},D_{\tau}^c)-(a+1)/m\geq(3a+2C_2-1)/m.\]
	This means that $A_m(t)\subset D_{\tau}$, and thus that $A_m(t)$ does not meet $\partial\Omega_\zeta$ by Lemma \ref{harmoniclemma2}(a). This means that we can replace $A_m(t)$ by $A_\zeta(t)$, which we partition as
	\[A_1=A_\zeta(t)\cap D_{t_0/m^2},\qquad A_2=A_\zeta(t)\cap B_{a/m}(z),\qquad A_3=A_\zeta(t)\setminus (A_1\cup A_2),\]
	where $t_0$ is chosen such that $D_{t_0/m^2}\subset(D_{t/m^2})_{\ell/m}$. By Lemma \ref{heleshawdist}, we can satisfy $d_H(D_{t_0/m^2},D_{t/m^2})\leq 2V\ell/mv$, or
	\[(D_{t/m^2})_{2V\ell/mv}\subset D_{t_0/m^2}.\]
	On the event $\mathcal{L}_{\ell/m}[t]^c$, no point in $(D_{t/m^2})_{\ell/m}$ is left out of $A_\zeta(t)$, so $A_1=\frac{1}{m}\mathbb{Z}^2\cap D_{t_0/m^2}$. Since $A_\zeta(t)$ has $t$ points and $A_1$ has at least $t-4\ell UVv^{-1}m$ points (using Lemma \ref{heleshawlength}), we know that $\#(A_2\cup A_3)\leq 4\ell UVv^{-1}m$. Noting that $H_\zeta(z')-H_\zeta(z_{m,i})\geq -1/mR_2$ for any $z'\in\Omega_\zeta$, this implies
	\[\sum_{z'\in A_3}\left(H_\zeta(z')-H_\zeta(z_{m,i(z')})\right)\geq -\frac{\#A_3}{mR_2}\geq -\frac{4\ell UVm}{mvR_2}=-\frac{4\ell UV}{vR_2},\]
	where $z_{m,i(z')}$ is the source point that initially generated the point $z'\in A_m(t)$. Next, we try to estimate the equivalent sum over $A_1$. By the discussion above, we know that only $4\ell UVv^{-1}m$ points can be outside the bounds of $A_1$, meaning that $\{z_{m,i(z')}\;|\;z'\in A_1\}$ differs from $\{z_{m,i}\;|\;0\leq i\leq t_0\}$ by at most $4\ell UVv^{-1}m$ points. Along with Lemma \ref{harmoniclemma2}(c), this implies
	\begin{align*}
		\sum_{z'\in A_1}(H_\zeta(z')-H_\zeta(z_{m,i(z')}))&\geq\sum_{z'\in A_1}H_\zeta(z')-\sum_{i=1}^{t_0}H_\zeta(z_i)-8\ell UVv^{-1}m\max|H_\zeta(z_{m,i})|\\
		&\geq-C'_2\log m-8UVK''\ell/v\\
		&\geq-C'_2a/C_3-8UVK''\ell/v.
	\end{align*}
	Adding up the contributions from $A_1$ and $A_3$ gives
	\begin{equation}\label{contributionseq}\sum_{z'\in A_1\cup A_3}(H_\zeta(z')-H_\zeta(z_{m,i(z')}))\geq-C'_2a/C_3-12UVK''\ell/v\geq -vba/12V,
	\end{equation}
	from the definitions of $C_3$ and $\ell$ above.
	
	Now, since $z\notin (D_{t/m^2})^{a/m}\supset(D_0)^{a/m}$, Lemma \ref{tentacles} tells us that
	\[\mathbb{P}\left(Q_{z,t}\cap\mathcal{L}_{\ell/m}[t]^c\cap\{\#A_2\leq ba^2\}\right)\leq C_0e^{-c_0a}\leq C_0e^{-3m^{2/5}}\]
	for large enough $m$, using the facts that $a\geq C_3 m^{2/5}$ and $C_3\geq 3/c_0$.
	
	On the event $Q_{z,t}$, the point $z$ is $a/m$ early but not $(a+1)/m$-early, so $a/m\leq d(z,D_{t/m^2})\leq (a+1)/m$. We know that $\zeta$ is the nearest point to $z$ in the annulus of Equation \ref{annulus}, which means that (for $a>2C_2$) we have $md(z,\zeta)\leq 5Va/v$. Then $F_\zeta(z)= v/5Va+O(a^{-2})$, and so by Lemma \ref{harmoniclemma}, for all $z'\in B(z,a)$,
	\[H_\zeta(z')-H_\zeta(z_{m,i(z')})\geq \frac{v}{5Va}-\frac{2}{mR_1}+O(a^{-2})\geq \frac{v}{6Va},\]
	as long as $m$ (and hence $a$) is large enough. On the event $\{\#A_2> ba^2\}$, this means
	\[\sum_{z'\in A_2}(H_\zeta(z')-H_\zeta(z_{m,i(z')}))\geq \frac{v\# A_2}{6Va}>\frac{vba}{6V},\]
	and hence (from Equation \ref{contributionseq}) on event $\{\#A_2> ba^2\}\cap\mathcal{L}_{\ell/m}[t]^c\cap Q_{z,t}$,
	\[M_\zeta(t)=\sum_{z'\in A_\zeta(t)}(H_\zeta(z')-H_\zeta(z_{m,i(z')}))>\frac{vba}{6V}-\frac{vba}{12V}=\frac{vba}{12V}.\]
	Thus, 
	\[\left\{M_\zeta(t)\leq vba/12V\right\}\subset\{\#A_2\leq ba^2\}\cup \left(\mathcal{L}_{\ell/m}[t]^c\cap Q_{z,t}\right)^c,\]
	and so
	\[\mathbb{P}(\{M_\zeta(t)<vba/12V\}\cap\mathcal{L}_{\ell/m}[t]^c\cap Q_{z,t})\leq\mathbb{P}\left(Q_{z,t}\cap\mathcal{L}_{\ell/m}[t]^c\cap\{\#A_2\leq ba^2\}\right)\leq C_0e^{-3m^{2/5}}.\]
\end{proof}
\begin{proof}[Step 3]
	Since $C_3\geq\frac{72V}{vb}$, we know that
	\[\frac{vba}{12V}\geq \frac{vbC_3}{12V}m^{2/5}\geq6m^{2/5}=M,\]
	with $M$ as in part 1. Using Lemma \ref{escapelemma2}, we find
	\[\mathbb{P}\left(\{S_\zeta\leq M\}\cap\{M_\zeta\geq vba/12V\}\right)\leq \mathbb{P}\left(\{S_\zeta\leq M\}\cap\{B_\zeta(S_\zeta)\geq M\}\right)\leq e^{-M/2}=m^{-3m^{2/5}}.\]
	Finally, we bound
	\begin{align*}
		\mathbb{P}(Q_{z,t}\cap\mathcal{L}_{\ell/m}[m^2s]^c)&\leq\mathbb{P}(Q_{z,t}\cap\{S_\zeta>M\})\\
		&\qquad+\mathbb{P}(Q_{z,t}\cap\{M_\zeta(t)\leq vba/12V\}\cap\mathcal{L}_{\ell/m}[t]^c)\\
		&\qquad+\mathbb{P}(\{S_\zeta\leq M\}\cap\{M_\zeta(t)\geq vba/12V\})\\
		&\leq (C_2+2)e^{-3m^{2/5}},
	\end{align*}
	which implies
	\begin{align*}
		\mathbb{P}(\Eps_{a/m}[m^2s]\cap\mathcal{L}_{\ell/m}[m^2s]^c)&=\sum_{t=1}^{\lfloor m^2s\rfloor}\sum_{z\in (D_0)^{ms}}\mathbb{P}(Q_{z,t}\cap\mathcal{L}_{\ell/m}[m^2s]^c)\\
		&\leq m^2s\cdot 10m^2(\op{diam}(D_0)+ms)^2 (C_2+2)e^{-3m^{2/5}}\leq m^{-2m^{2/5}}.
	\end{align*}
\end{proof}
	\section{Late points imply early points}
Very roughly, we would like the proof of the second part of Theorem \ref{biggun'} to go as follows. If $\zeta$ is the first $(\ell/m)$-late point in $A_\zeta(t)$, then at the time $T\sim m^2\tau+m\ell$, the set $A_\zeta(t)$ has several particles at every boundary point $z\neq\zeta$ in $\partial \Omega_\zeta$. Since $H_\zeta(\zeta)$ is much larger than $H_\zeta(z\neq\zeta)$, this would tell us in turn that $M_\zeta$ would have a much lower value than expected. Combined with Lemmas \ref{insideest} and \ref{outsideest}, we would be able to recover a strong upper bound of the probability of $\mathcal{L}_{\ell/m}[T]\cap \mathcal{E}_{a/m}[T]^c$.

Unfortunately, we are unable to say that the difference $H_\zeta(z)-H_\zeta(z_{m,i})$ that occurs in the expression for $M_\zeta$ is even negative, let alone a large negative number. The problem that occurs in the general source (i.e., non-disk) setting is that we cannot obtain a positive lower bound on $H_\zeta(z_{m,i})$, as the source point $z_{m,i}$ may be ``behind'' the pole $\zeta$, as shown in Figure \ref{behindfig}. 
\begin{figure}[H]
	\centering
	\begin{tikzpicture}
		\node[anchor=south west,inner sep=0] (image) at (0,0) {\includegraphics[scale=.36]{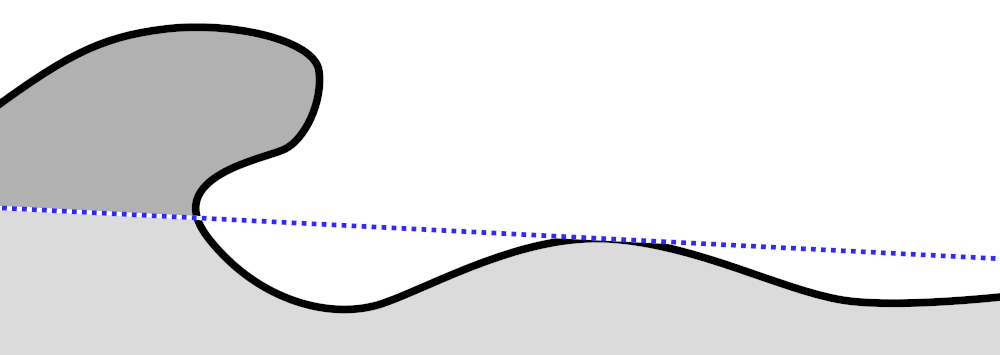}};
		\node[text width=2cm,rotate=0] at (8.6,1.8) {$\zeta$};
		\node[text width=2cm,rotate=0] at (2.15,1) {$\mathbf{H_\zeta<0}$};
		\node[text width=2cm,rotate=0] at (2.15,2.5) {$\mathbf{H_\zeta>0}$};
	\end{tikzpicture}
	\caption{The original harmonic $H_\zeta$ is negative on a half-plane cut out by the pole $\zeta$. While this does not come into play in the case of a point-source, it is critical in extended-source IDLA, forcing us to define a new harmonic function $\tilde{H}_\zeta$ to continue with the proof.}\label{behindfig}
\end{figure}
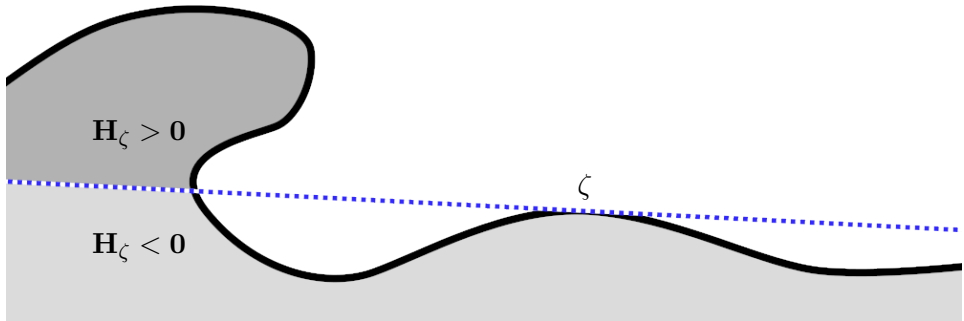

To remedy this issue, we introduce a second harmonic function $\tilde{H}_\zeta$, defined to be the discrete Poisson kernel on a slightly modified domain $\tilde{\Omega}_\zeta\approx \Omega_\zeta$. We will see that the difference $\tilde{H}_\zeta(z)-\tilde{H}_\zeta(z_{m,i})$ is negative and bounded away from zero, so our program will go through roughly as mentioned above.

On the other hand, we will not be able to get a strong replacement for Lemma \ref{harmoniclemma}(c), which tells us that $H_\zeta$ closely approximates a continuum harmonic function. This leads to an overall $m^{2/5}$ error---rather than the logarithmic errors we saw in Lemma \ref{harmoniclemma2}(c)---when summing $\tilde{H}_\zeta$ over the set $D_s$, and it eventually creates the $m^{-3/5}$ error of Theorem \ref{biggun'}.

\subsection{The Poisson kernel on $\Omega_\zeta$}
We introduce a new, positive harmonic function on the new set
\[\tilde{\Omega}_\zeta=\left(D_\tau\cap\frac{1}{m}\mathbb{Z}^2\right)\setminus\{\zeta\pm iu\;|\;u\in(0,1/m^2)\}.\]
Namely, if $W_z(t)$ is a (grid) Brownian motion in $\tilde{\Omega}_\zeta$ starting at $z$, and $\tau^*$ is the first exit time of $W_\zeta(t)$ from $\tilde{\Omega}_\zeta$, we define
\[\tilde{H}_\zeta(z):=\mathbb{P}[W_z(\tau^*)=\zeta].\]
We can recognize this as the Poisson kernel associated to the set $\tilde{\Omega}_\zeta$. In particular, it satisfies the following key properties:
\begin{lemma}\label{harmoniclemma2-1}For any $m$, $\tilde{H}_\zeta$ satisfies the following:
	\begin{enumerate}[label=\emph{(\alph*)}]
		\item $\tilde{H}_\zeta$ is grid harmonic in $\tilde{\Omega}_\zeta$, and $H_\zeta\geq 0$.
		\item $\tilde{H}_\zeta(\zeta)=1$. For all $z\in\partial\tilde{\Omega}_\zeta\setminus\{\zeta\}$, we have $\tilde{H}_\zeta(z)=0$.
		\item For any $U\subset\tilde{\Omega}_\zeta$ with $md(\zeta,U)>C_2$,
		\[\tilde{H}_\zeta(z)\leq\frac{1}{2mR_0}+\frac{1}{md(\zeta,U)-C_2}.\]
		\item Let $\zeta'=\zeta-1/m\in\tilde{\Omega}_\zeta$. Then
		\[\tilde{H}_\zeta(z)=c_\zeta G_{\tilde{\Omega}_\zeta}(\zeta',z)\]
		on $\tilde{\Omega}_\zeta\setminus B_{1/m}(\zeta)$, where $1/16\leq c_\zeta\leq 1$ and
		\[G_{\tilde{\Omega}_\zeta}(y,z):=\mathbb{E}g(W_z(\tau^*)-y)-g(z-y)\]
		is the Green's function associated to $\tilde{\Omega}_\zeta$.
	\end{enumerate}
\end{lemma}
\begin{proof}$ $\medskip
	
	\noindent (a,b) The first two points follow from the definition of $\tilde{H}_\zeta$.
	
	\noindent(c) From \ref{harmoniclemma}(a,d), we know that 
	\[H_\zeta+\frac{1}{2mR_0}\geq 0\]
	on all of $\tilde{\Omega}_\zeta\subset\Omega_\zeta$, and that $H_\zeta(\zeta)+\frac{1}{2mR_0}\geq 1$. In particular, 
	\[H_\zeta+\frac{1}{2mR_0}\geq \tilde{H}_\zeta\]
	on the boundary of $\tilde{\Omega}_\zeta$, so we know from the maximum principle and Lemma \ref{harmoniclemma2}(b) that
	\[\tilde{H}_\zeta|_U\leq \left[H_\zeta+\frac{1}{2mR_0}\right]_U\leq\frac{1}{2mR_0}+\frac{1}{md(\zeta,U)+C_2}.\]
	
	\noindent(d) This follows from the last-exit decomposition for simple random walks \cite[Prop.~4.6.4]{lawler2010random}.
\end{proof}

\begin{lemma}\label{harmoniclemma2-2}Suppose $D_s$ is smooth. Then,
	\begin{enumerate}[label=\emph{(\alph*)}]
		\item For any $z\in D_\tau\cap\Omega_\zeta$,
		\[\left|G_{\tilde{\Omega}_\zeta}(\zeta',z)-G_{D_\tau}(\zeta',z)\right|\leq\frac{C_2}{m^2d(z,\partial\tilde{\Omega}_\zeta)^2}+\frac{C_2}{m^2d(z,\zeta')^3},\]
		where $G_{D_\tau}$ is the continuous Green's function of $D_\tau$.
		\item For any $z\in D_\tau$,
		\[\left|G_{D_\tau}(\zeta',z)-\frac{c'_\zeta}{m}J_\zeta(z)\right|\leq\frac{C_2}{m^2d(z,\zeta')^2}+\frac{C_2}{m^2d(z,\zeta)^2},\]
		where $c'_\zeta\in[2^{-1/2},1]$ depends only on $\zeta$ and $J_{D_\tau}$ is the Poisson kernel on $D_\tau$.
		\item The following mean-value property holds:
		\[\left|\sum_{z\in \tilde{\Omega}_\zeta\cap\frac{1}{m}\mathbb{Z}^2}\tilde{H}_\zeta(z)-\sum_{i=1}^{\lfloor m^2\tau\rfloor}\tilde{H}_\zeta(z_{m,i})\right|\leq C'_2m^{2/5}.\]
	\end{enumerate}
\end{lemma}
\begin{proof}
	\noindent (a) For this, we use the estimate
	\[g(x,y)=\log m|x-y|+\lambda+O\left(\frac{1}{m^2|x-y|^2}\right)\]
	mentioned in Section \ref{kernelsection}. This implies
	\[G_{\Omega_\zeta}(z,\zeta')=\mathbb{E}g(W_{\zeta'}(\tau^*),z)-g(\zeta',z)=\mathbb{E}\log|W_{\zeta'}(\tau^*)-z|-\log|\zeta'-z|+O\left(\frac{1}{m^2d(z,\partial\Omega_\zeta)^2}\right),\]
	as the $\log m$ and $\lambda$ terms cancel out. Fixing $z$, we see that $\mathbb{E}\log|W_{z'}(\tau^*)-z|$ is a discrete harmonic function of $z'$, with boundary values $\log|z'-z|$ for $z'\in\partial\tilde{\Omega}_\zeta$. With the possible exception of the points $\zeta\pm im^{-2}$, all boundary points of $\tilde{\Omega}_\zeta$ also lie on the boundary of $\partial D_\tau$; then we can compare $f_0(z')=\mathbb{E}\log|W_{z'}(\tau^*)-z|$ with the continuous harmonic function $f_1(z')=G_{D_\tau}(z',z)+\log|z'-z|$. Indeed, the latter has fourth derivative bounded above by $C/d(z,z')^3$, so we know 
	\begin{align*}
		\left|\Delta_h \left(\mathbb{E}\log|W_{z'}(\tau^*)-z|-(G_{D_\tau}(z',z)+\log|z'-z|)\right)\right|&=\left|\Delta_h \left(G_{D_\tau}(z',z)+\log|z'-z|\right)\right|\\
		&\leq \frac{C}{m^2d(z,z')^3},
	\end{align*}
	where $\Delta_h$ is the five-point stencil Laplacian. Furthermore, $\mathbb{E}\log|W_{z'}(\tau^*)-z|$ and $G_{D_\tau}(z',z)+\log|z'-z|$ differ by at most $O(m^{-2})$ on the boundary (at $\zeta\pm im^{-2}$), so the maximum principle gives
	\[\left|\mathbb{E}\log|W_{z'}(\tau^*)-z|-(G_{D_\tau}(z',z)+\log|z'-z|)\right|\leq\Delta_h^{-1}\left(\frac{C}{m^2d(z,\zeta')^3}\right)+O(m^{-2}).\]
	The claim follows, as $\Delta_h^{-1}\left(\frac{C}{m^2d(z,\zeta')^3}\right)$ approaches $\zeta'$ no faster than $O(z^{-3})$.
	
	\noindent (b) This follows from the general formula $J_\zeta(z)=\partial_\unit{n}G_{D_\tau}(z',z)|_{z'=\zeta}$, along with the fact that $\zeta'$ is at most an angle $\pi/4$ away from the normal direction inwards from $\zeta$.
	
	\noindent (c) Set $\alpha_0=m^{-2/5}$ and $\eps_0=m^{-1/5}$, and let $B^+_\zeta\subset D_\tau$ and $B^-_\zeta\subset D_\tau^c$ be the disks of radius $R_0$ tangent to $\partial D_\tau$ at $\zeta$. For each $\alpha>0$, we partition $\ol{\Omega}_\zeta:=\tilde{\Omega}_\zeta\cap\frac{1}{m}\mathbb{Z}^2$ by sets $A^\alpha$, $B^\alpha$, $C^\alpha$, and $D^\alpha$ as follows:
	\[A^\alpha=\left\{z\in\ol{\Omega}_\zeta\;\big|\;d(z,\partial \Omega_\zeta)>\alpha\right\},\qquad B^\alpha=\ol{\Omega}_\zeta\setminus (A^\alpha\cup B_{\eps_0}(\zeta)),\]
	\[C^\alpha=\ol{\Omega}_\zeta\setminus (A^\alpha\cup B^\alpha\cup B^+_\zeta),\qquad D^\alpha=\ol{\Omega}_\zeta\setminus (A^\alpha\cup B^\alpha\cup C^\alpha).\]
	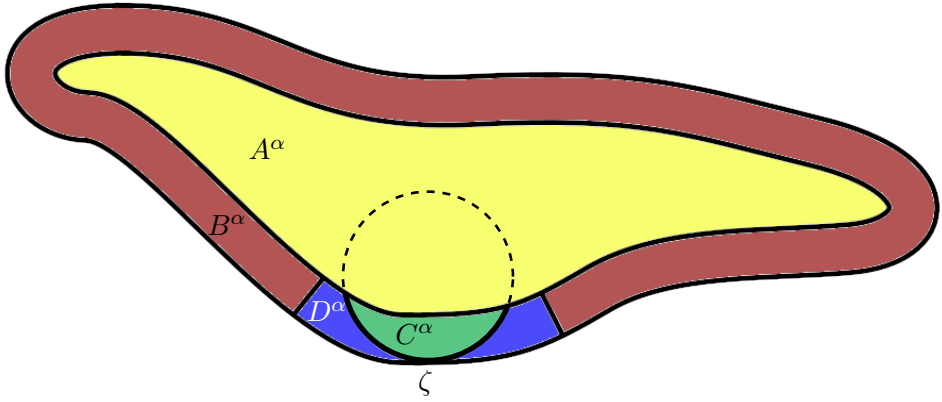
\begin{figure}[H]
		\centering
		\begin{tikzpicture}
			\node[anchor=south west,inner sep=0] (image) at (0,0) {\includegraphics[scale=.4]{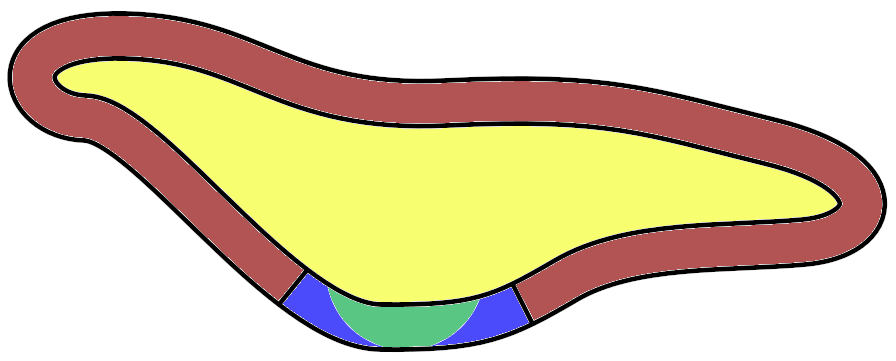}};
			\draw [dashed, line width=1pt] (5.73,1.31) circle(1.13cm);
			\draw [line width=2pt] (4.63,1.09) arc(192:340:1.13cm);
			\node[text width=2cm,rotate=0] at (6.6,-.1) {$\zeta$};
			\node[text width=2cm,rotate=0] at (4.35,3) {$A^\alpha$};
			\node[text width=2cm,rotate=0] at (3.8,2) {$B^\alpha$};
			\node[text width=2cm,rotate=0] at (6.3,.55) {$C^\alpha$};
			\node[text width=2cm,rotate=0,color=white] at (5.12,.85) {$D^\alpha$};
		\end{tikzpicture}
		\caption{An illustration of our partition $A^\alpha\cup B^\alpha\cup C^\alpha\cup D^\alpha$ of $\ol{\Omega}_\zeta$, with the circle $\partial B^+_\zeta$ marked. }\label{breakdownfig}
	\end{figure}
	We will bound the error $\left|\sum G_{\tilde{\Omega}_\zeta}(\zeta',\cdot)-m^2\int G_{D_\tau}(\zeta',\cdot)\right|$ over each of these sets (with $\alpha=\alpha_0$) in turn. For any $\alpha>C_2/m$, we know that $A^\alpha\subset D_\tau$ from Lemma \ref{harmoniclemma}(a), so part (a) implies \[|G_{\Omega_\zeta}(\zeta',z)-G_{D_\tau}(\zeta',z)|\leq C_2/(m\alpha)^2\]
	on $A^\alpha$. This allows us to bound 
	\begin{align*}
		\left|\sum_{A^{\alpha_0}}G_{\tilde{\Omega}_\zeta}(\zeta',z)-m^2\int_{A^{\alpha_0}}G_{D_\tau}(\zeta',z)\right|&\leq 2m^2\int_{\alpha_0}^\infty d\alpha\;\frac{C_2}{m^2\alpha^2}\op{Len}(\partial A^\alpha)+2\int_{(D_\tau)_{\alpha_0}}\frac{C_2}{d(z,\zeta')}\\
		&\leq 2\int_{\alpha_0}^\infty d\alpha\;\frac{C_2U}{\alpha^2}\sqrt{1+\tau}+\tfrac{1}{2}C_2'\alpha_0^{-1}\\
		&\leq C_2'\alpha_0^{-1}=C_2'm^{2/5},
	\end{align*}
	with $C_2'$ large enough, using the fact that $\op{Len}(\partial A^\alpha)\leq\op{Len}(D_\tau)$, and using $|\nabla G_{D_\tau}(z)|=O(|z-\zeta|^{-2})$ to relate the initial integral to a sum. More precisely, we could integrate $C_2/d(z,\zeta')$ over an encompassing shape as in Figure \ref{circfig} to retrieve the bound $\int_{(D_\tau)_{\alpha_0}}\frac{C_2}{d(z,\zeta')}=O(\alpha_0^{-1})$. This immediately gives
	\begin{align*}
		\left|\sum_{A^{\alpha_0}}\tilde{H}_\zeta'(z)-c_\zeta c'_\zeta m\int_{A^{\alpha_0}}J_\zeta(z)\right|\leq \tfrac{1}{4}C_2'm^{2/5},
	\end{align*}
	from Lemma \ref{harmoniclemma2-1}(d) and part (b) above.
	
	Next, we control the sum over $B^{\alpha_0}$. Since $\partial D_\tau$ is smooth, the probability of a point $z$ near the boundary to exit $\tilde{\Omega}_\zeta$ at $\zeta$ is bounded by $C\frac{md(z,\partial\tilde{\Omega}_\zeta)}{m^2d(z,\zeta)^2}$, which we estimate as
	\[\sum_{B^{\alpha_0}}\tilde{H}_\zeta(z)\leq2Cm^2\cdot\int_{B^{\alpha_0}}\frac{md(z,\partial\tilde{\Omega}_\zeta)}{m^2d(z,\zeta)^2}\leq \tfrac{1}{2}C'_2m\int_{0}^{\alpha_0}\alpha/\eps_0=\tfrac{1}{4}C'_2m\alpha_0^2/\eps_0=\tfrac{1}{4}C'_2m^{2/5}.\]
	
	For the remaining sets, we introduce slice coordinates $(x,y)$ for $\partial D_\tau$ near $\zeta$, such that $\zeta=(0,0)$. These points are bounded outside the disk $B^-_\zeta$, so the probability of their associated random walks exiting $\tilde{\Omega}_\zeta$ at $\zeta$ is bounded by
	\[\tilde{H}_\zeta(z)\leq\mathbb{P}[W_z\;\text{enters}\;B^-_\zeta\;\text{at}\;\zeta]\leq \frac{1}{m}\cdot\frac{1}{x^2+y^2}(y+O(x^2+y^2)).\]
	Then we find
	\begin{align*}
		\sum_{C^{\alpha_0}} \tilde{H}_\zeta&\leq m^2\cdot\frac{4}{m}\int_0^{\eps_0}dx\int_0^{R_0^{-1}x^2}dy\;\left(\frac{y}{x^2+y^2}+C\right)\\
		&\leq 4m\int_0^{\eps_0}dx\;\left(\log(1+R_0^{-1}x^2)+CR_0^{-1}x^2\right)\\
		&\leq \tfrac{3}{4}mC_2'\int_0^{\eps_0}dx\;x^2=\tfrac{1}{4}C'_2m\eps_0^3=\tfrac{1}{4}C'_2m^{2/5},
	\end{align*}
	and similarly for $c_\zeta c'_\zeta m\int_{C^{\alpha_0}}J_\zeta(z)$.
	
	Finally, for $z\in D^{\alpha_0}$, we can bound
	\[\mathbb{P}[W_z\;\text{exits}\;B^+_\zeta\;\text{at}\;\zeta]\leq\tilde{H}_\zeta(z)\leq\mathbb{P}[W_z\;\text{enters}\;B^-_\zeta\;\text{at}\;\zeta],\]
	which gives
	\[\tilde{H}_\zeta(z)=\frac{1}{m}\cdot\frac{1}{x^2+y^2}(y+O(x^2+y^2))\]
	and similarly 
	\[\frac{c_\zeta c_\zeta'}{m}J_\zeta(z)=\frac{1}{m}\cdot\frac{1}{x^2+y^2}(y+O(x^2+y^2)).\]
	The error is dominated by the second order part:
	\[\left|\sum_{D^{\alpha_0}}\tilde{H}_\zeta'(z)-c_\zeta c'_\zeta m\int_{D^{\alpha_0}}J_\zeta(z)\right|\leq 4m\int_0^{\eps_0}dx\int_0^{\alpha_0}dy\;C\frac{x^2+y^2}{x^2+y^2}=4mC\eps_0\alpha_0=\tfrac{1}{4}C'_2m^{2/5}.\]
	The proof finishes as does Lemma \ref{harmoniclemma2}(c), but including the extra factor $c_\zeta$ from Lemma \ref{harmoniclemma2-1}(d).
\end{proof}

Just as with $H_\zeta$, we associate the following martingale to $\tilde{H}_\zeta$:
\[\tilde{M}_\zeta(t):=\sum_{\ell=0}^{\lfloor t\rfloor-1}\left(\tilde{H}_\zeta(\beta_\ell(1))-\tilde{H}_\zeta(z_{m,\ell})\right)+\tilde{H}_\zeta(\beta_\ell(t-\lfloor t\rfloor))-\tilde{H}_\zeta(z_{m,\lfloor t\rfloor}),\]
using the same notation as in Section \ref{martingalesection}. Now, the rescaled function $(1+1/2mR_0)\tilde{H}_\zeta-1/2mR_0$ satisfies the properties outlined in Lemmas \ref{harmoniclemma}(a,b,d) and Lemma \ref{harmoniclemma2}(b), so we can prove the following parallels to Lemmas \ref{insideest} and \ref{outsideest} exactly as before:
\begin{lemma}\label{insideest2}
	Suppose $D_s$ is a smooth flow arising from an initial mass distribution. For 
	\[m\geq \max(3a+C_2,2C_2/\inf\nolimits_\zeta R_1),\]
	all $t\geq 1$, and $\zeta\notin (D_{t/m^2})^{(4a+2C_2)/m}$, we have
	\[\mathbb{E}\left[e^{\tilde{S}_\zeta(t)}\mathbb{1}_{\Eps_{(a+1)/m}(t)^c}\right]\leq m^K,\]
	where $\tilde{S}_\zeta(t)=\langle \tilde{M}_\zeta,\tilde{M}_\zeta\rangle_t$.
\end{lemma}

\begin{lemma}\label{outsideest2}
	Suppose $D_s$ is smooth, and fix $a\geq 2C_2+2$, $\ell\leq a$, and $t>0$. For 
	\[m\geq \max(3a+C_2,5a/\inf\nolimits_\zeta R_1)\]
	and $\zeta\in\frac{1}{m}\mathbb{Z}^2\cap \left((D_{t/m^2})_{\ell/m}\setminus D_{0}\right)$, we have
	\[\mathbb{E}\left[e^{\tilde{S}_\zeta(t)}\mathbb{1}_{\Eps_{(a+1)/m}(t)^c}\right]\leq m^Ke^{K'a}.\]
\end{lemma}
	\subsection{Second estimate}\label{secondestsection}

\begin{lemma}\label{secondest}
	There is an absolute constant $C_4>0$ such that, for large enough $m$, if $s\in[0,T]$, $\ell\geq C_4 m^{2/5}$, and $a\leq\ell^2/C_4m^{2/5}$, then
	\[\mathbb{P}(\mathcal{L}_{\ell/m}[m^2s]\cap\Eps_{a/m}[m^2s]^c)\leq e^{-2m^{2/5}}.\]
\end{lemma}
\begin{proof}
	Without loss of generality, let $a=\ell^2/C_4m^{2/5}\geq\ell$. We can further suppose that $m\geq \max(3a+C_2,5a/\inf\nolimits_\zeta R_1)$. Indeed, otherwise we have $a=\eps m$ for a constant $\eps=\inf(1/4,\inf R_1/5)$; by Lemma \ref{levinepereslemma}, we know that we can choose $m$ large enough that
	\[\mathbb{P}(\mathcal{L}_{\ell/m}[m^2s]\cap\Eps_{a/m}[m^2s]^c)\leq \mathbb{P}(\mathcal{L}_{\ell/m}[m^2s])\leq C_0e^{-c_0m^2/\log m}\leq e^{-2m^{2/5}}.\]
	
	Fix $\zeta\in\frac{1}{m}\mathbb{Z}^2\cap((D_s)_{\ell/m}\setminus D_0)$ and set $T_1\leq m^2s$ minimal such that $\zeta\in(D_{T_1})_{\ell/m}$. Then we know that $d(\zeta,\partial D_{T_1})=\ell/m$---by Lemma \ref{heleshawdist}, this implies that
	\[
	T_1-m^2\tau\geq 2m^2\sqrt{1+\tau}(\sqrt{1+T_1/m^2}-\sqrt{1+\tau})\geq \frac{2m\ell}{V}.
	\]
	Let $L[\zeta]=\{\zeta\notin A_m(m^2T_1)\}$ be the event that $\zeta$ is $(\ell/m)$-late. Then
	\[\mathcal{L}_{\ell/m}[m^2s]=\bigcup\nolimits_\zeta L[\zeta].\]
	On the event $L[\zeta]$, we know that any particles in $A_\zeta(T_1)$ that hit the boundary must do so away from $\zeta$; that is, $\tilde{H}_\zeta\equiv 0$ for these particles. As in \cite{10.2307/23072157}, this implies that $\tilde{M}_\zeta(T_1)$ is maximized if the interior of $\tilde{\Omega}_\zeta\cap\frac{1}{m}\mathbb{Z}^2$ is fully occupied by $A_\zeta(T_1)$, so we can bound $\tilde{M}_\zeta(T_1)$ as follows:
	\begin{align*}
		\tilde{M}_\zeta(T_1)\leq \sum_{z\in \partial\tilde{\Omega}_\zeta\cap A_\zeta(T_1)}\left(0-H_\zeta(z_{m,i(z)})\right)+\sum_{z\in \ol{\Omega}_\zeta}\left(\tilde{H}_\zeta(z)-\tilde{H}_\zeta(z_{m,i(z)})\right),
	\end{align*}
	where $z_{m,i(z)}$ is the source point from which the particle landing at $z\in\ol{\Omega}_\zeta$ started, and weighting each term of the first sum by its number of occurrences in the multiset $A_\zeta(T_1)$. First, we reorganize the source terms of the two sums:
	\begin{align*}
		\tilde{M}_\zeta(T_1)&\leq -\sum_{i=m^2\tau+1}^{T_1}\tilde{H}_\zeta(z_{m,i})+\sum_{z\in \ol{\Omega}_\zeta}\tilde{H}_\zeta(z)-\sum_{i=1}^{m^2\tau}\tilde{H}_\zeta(z_{m,i})\\
		&\leq -(T_1-m^2\tau)\inf\nolimits_{i}\tilde{H}_\zeta(z_{m,i})+\sum_{z\in \ol{\Omega}_\zeta}\tilde{H}_\zeta(z)-\sum_{i=1}^{m^2\tau}\tilde{H}_\zeta(z_{m,i})\\
		&\leq -2m\ell V^{-1}\inf\nolimits_{i}\tilde{H}_\zeta(z_{m,i})+\sum_{z\in \ol{\Omega}_\zeta}\tilde{H}_\zeta(z)-\sum_{i=1}^{m^2\tau}\tilde{H}_\zeta(z_{m,i})\\
		&\leq -2c\ell+\sum_{z\in \ol{\Omega}_\zeta}\tilde{H}_\zeta(z)-\sum_{i=1}^{m^2\tau}\tilde{H}_\zeta(z_{m,i})
	\end{align*}
	for a constant $c>0$ depending only on the flow, using Lemmas \ref{harmoniclemma2-2}(a,b) to deduce that $\inf\nolimits_{i}\tilde{H}_\zeta(z_{m,i})=\Theta(m^{-1})$. Next, notice that the two right-hand sums are the same that appear in Lemma \ref{harmoniclemma2-2}(c), implying that
	\begin{align*}
		\tilde{M}_\zeta(T_1)\leq -2c\ell+C_2'm^{2/5}\leq -c\ell,
	\end{align*}
	so long as $C_4\geq C_2'/c$.
	
	Choose $m$ large enough that $e^{a}\geq e^{C_4m^{2/5}}\geq m$. By Lemma \ref{outsideest2},
	\[\mathbb{E}\left[e^{\tilde{S}_\zeta(m^2T_1)}\mathbf{1}_{\Eps_{(a+1)/m}(m^2T_1)^c}\right]\leq m^Ke^{K'a}\leq e^{(K+K')a}.\]
	Let $M=(K+K'+1)a$, so Markov's inequality implies
	\[\mathbb{P}\left(\{\tilde{S}_\zeta(m^2T_1)>M\}\cap\Eps_{a/m}[m^2T_1]^c\right)\leq e^{-M}\mathbb{E}\left[e^{\tilde{S}_\zeta(m^2T_1)}\mathbf{1}_{\Eps_{(a+1)/m}(m^2T_1)^c}\right]\leq e^{-a}.\]
	Since $\tilde{M}_\zeta(m^2T_1)\leq -c\ell$ on the event $L[\zeta]$, this means that
	\begin{align*}
		\mathbb{P}\left(\Eps_{a/m}[m^2T_1]^c\cap L[\zeta]\right)&\leq \mathbb{P}\left(\{\tilde{S}_\zeta(m^2T_1)>M\}\cap\Eps_{a/m}[m^2T_1]^c\right) \\
		&\qquad+ \mathbb{P}\left(\{\tilde{S}_\zeta(m^2T_1)\leq M\}\cap L[\zeta]\right)\\
		&\leq \mathbb{P}\left(\{\tilde{S}_\zeta(m^2T_1)>M\}\cap\Eps_{a/m}[m^2T_1]^c\right) \\
		&\qquad+ \mathbb{P}\left\{\tilde{S}_\zeta(m^2T_1)\leq M,\tilde{M}_\zeta(m^2T_1)\leq -c\ell\right\}\\
		&\leq e^{-a}+ e^{-c^2\ell^2/2M}\\
		&=e^{-a}+ e^{-c^2C_4m^{2/5}/2(K+K'+1)}\leq 2e^{-4m^{2/5}}\leq e^{-3m^{2/5}}
	\end{align*}
	for $C_4\geq 4(K+K'+1)/c^2$, using Lemma \ref{escapelemma2} and the fact that $\tilde{M}_\zeta(t)=\tilde{B}_\zeta(\tilde{S}_\zeta(t))$ for a centered Brownian motion $\tilde{B}_\zeta$. We conclude that
	\begin{align*}
		\mathbb{P}\left(\mathcal{L}_{\ell/m}[m^2s]\cap\Eps_{a/m}[m^2s]^c\right)&\leq \sum_{\zeta\in\frac{1}{m}\mathbb{Z}^2\cap((D_s)_{\ell/m}\setminus D_0)}\mathbb{P}\left(\Eps_{a/m}[m^2T_1]^c\cap L[\zeta]\right)\\
		&\leq 2\op{Vol}(D_s)m^2 e^{-3m^{2/5}}\leq e^{-2m^{2/5}}.
	\end{align*}
\end{proof}

\section{Proof of Theorem \ref{biggun'}}
Choose $m$ large, $\eps\leq \alpha/4$, and $s\in[0,T]$. From Lemma \ref{levinepereslemma}, we know
\[\mathbb{P}(\mathcal{L}_\eps[m^2s])\leq C_0e^{-c_0m^2/\log m}\leq e^{-2m^{2/5}}.\]
Set $\ell_0=\eps m$, and define values $a_k,b_k$ as follows:
\[a_k=\alpha^{-1}\ell_k,\qquad \ell_k=\sqrt{C_4m^{2/5}a_{k-1}},\]
where $\alpha,C_4>0$ are as in Sections \ref{firstestsection} and \ref{secondestsection}, respectively. Now, if $\ell_k\geq \alpha^{-1}C_4m^{2/5}$, we know that
\[a_{k-1}=\frac{\ell_k^2}{C_4m^{2/5}}\geq \alpha^{-2}C_4m^{2/5}\]
and thus that
\[\ell_{k-1}=\alpha a_{k-1}\geq\alpha^{-1}C_4m^{2/5}.\]
Thus, if $\ell_k\geq \alpha^{-1}C_4m^{2/5}$, we know that $\ell_n\geq C_4m^{2/5}$ and that $a_n\geq C_3m^{2/5}$ for all $n\leq k$, assuming without loss of generality that $\alpha^{-2}C_4\geq C_3$. We also know (from the choice of $\eps$) that $a_0\leq m/4$; in general, if $a_k\leq m/4$, then
\[a_{k+1}=\alpha^{-1}\sqrt{C_4m^{2/5}a_k}\leq \frac{1}{2}\alpha^{-1}\sqrt{C_4m^{7/5}}\leq m/4\]
for large enough $m$. Then the pair $(\ell_n,a_n)$ satisfies the hypothesis of Lemma \ref{firstest}, and similarly for $(\ell_n,a_{n-1})$ and Lemma \ref{secondest}. By induction, this implies
\[\mathbb{P}(\mathcal{E}_{a_k/m}[m^2s])\leq \mathbb{P}(\mathcal{L}_{\ell_k/m}[m^2s]^c\cap \mathcal{E}_{a_{k}/m}[m^2s]) + \mathbb{P}(\mathcal{L}_{\ell_k/m}[m^2s])\leq (2k+2)e^{-2m^{2/5}}\]
and
\[\mathbb{P}(\mathcal{L}_{\ell_k/m}[m^2s])\leq \mathbb{P}(\mathcal{L}_{\ell_k/m}[m^2s]\cap \mathcal{E}_{a_{k-1}/m}[m^2s]^c) + \mathbb{P}(\mathcal{E}_{a_{k-1}/m}[m^2s])\leq (2k+1)e^{-2m^{2/5}},\]
so long as $\ell_k\geq \alpha^{-1}C_4m^{2/5}$.

Now, set $A=\alpha^{-1}C_4m^{2/5}$, so that $\ell_k=A^{1-2^{-k}}\ell_0^{2^{-k}}$ and $m_k=A^{1-2^{-k}}m_0^{2^{-k}}$. With this formula, we see that the first time $\ell_k<2A$ occurs is when
\[k=\lceil \log_2\log_2(A^{-1}\ell_0)\rceil\leq c\log\log m,\]
for some $c>0$ independent of $m$. Fix $k'=\lceil \log_2\log_2(A^{-1}\ell_0)\rceil$; iterating $k'$ times, the above calculation shows that
\[\mathbb{P}(\mathcal{L}_{\alpha^{-1}C_4m^{-3/5}}[m^2s])\leq\mathbb{P}(\mathcal{L}_{\ell_{k'}/m}[m^2s])\leq (2k'+1)e^{-2m^{2/5}}\leq \frac{1}{2}e^{-m^{2/5}}\]
and
\[\mathbb{P}(\mathcal{E}_{\alpha^{-2}C_4m^{-3/5}}[m^2s])\leq\mathbb{P}(\mathcal{E}_{a_{k'}/m}[m^2s])\leq (2k'+2)e^{-2m^{2/5}}\leq \frac{1}{2}e^{-m^{2/5}}.\]
Set $C_5=\alpha^{-2}C_4$. Putting these bounds together, we get
\begin{align*}
	\mathbb{P}\bigg\{(D_{\tau})_{C_5m^{-3/5}}\cap\frac{1}{m}\mathbb{Z}^2&\subset A_m(m^2\tau)\subset (D_{\tau})^{C_5m^{-3/5}}\;\text{for all}\;\tau\in[0,s]\bigg\}^c\\
	&\leq \mathbb{P}(\mathcal{L}_{C_5m^{-3/5}}[m^2s])+\mathbb{P}(\mathcal{E}_{C_5m^{-3/5}}[m^2s])\leq e^{-m^{2/5}}.
	\tag*{\qed}
\end{align*}

	\section{Concluding Remarks}\label{conclusion}
There are a number of possible improvements to the results proven here. Most importantly, it would be interesting to improve the $m^{-3/5}$ bounds on the fluctuations; we expect that fluctuations are truly of order $m^{-1}\log m$, as in the point-source case. Hypothetically, this result could be proven using our technique---the primary obstacle is that we need a stronger version of Lemma \ref{harmoniclemma2}(c), which quantifies how closely $\tilde{H}_\zeta$ approximates a continuum harmonic function. In general, if we can replace the $m^{2/5}$ in \ref{harmoniclemma2}(c) with $m^\eps$ [resp., $\log m$], we could derive bounds on the fluctuations of order $m^{-1+\eps}$ [resp., $m^{-1}\log m$]. On the flip side, this also means that we could significantly weaken both \ref{harmoniclemma}(c) and \ref{harmoniclemma2}(c) and still prove a non-trivial convergence rate of IDLA.

Furthermore, it would be interesting to lift some of the hypotheses we set on the flow. However, we imagine that it is less likely our technique would apply without the requirements of a concentrated mass distribution or a smooth flow. Indeed, both hypotheses are necessary to guarantee that $R_0$ is bounded away from 0, and thus that $H_\zeta$ is small enough on the boundary. However, if an independent bound on $\tilde{H}_\zeta$ could be obtained, showing that it satisfies $\tilde{H}_\zeta(z)\leq\frac{1}{md(z,\zeta)-C_2}$ without comparing it to $H_\zeta$, it could be used in place of $H_\zeta$ for both parts of the proof.

There are also closely related settings that have not been studied extensively. An interesting example would be to replace the ``solid'' initial sets $Q^s_i$ with submanifolds of $D_0$. Since these would be zero volume, they could eject particles evenly from all points rather than having a moving source $\bigsqcup_{i}\partial Q^{s_i}_i$. Another example would be a collection of point sources; in fact, the theorem corresponding to Lemma \ref{levinepereslemma} in this setting has already been proved by Levine and Peres \cite[Theorem 1.4]{Levine_2010}, so it would likely not be too difficult to adapt our argument to this case.

Finally, a question we will investigate in the sequel is that of the scaling limits of the fluctuations themselves. Jerison, Sheffield, and Levine \cite{jerison2014} studied this question for same-time fluctuations in the point-source case, and they found that, when the fluctuations are scaled up by a factor of $m^{d/2}$ (in dimension $d$), they have a weak limit in law of a certain Gaussian random distribution. They found a similar result in the case of a discrete cylinder $\mathbb{Z}\times\mathbb{Z}/m\mathbb{Z}$ with source points along a fixed-height circumference \cite{jerison2013internal}; here, they further studied the correlations between fluctuations at different times in the flow. The same question has been studied by Eli Sadovnik \cite{eli16} in the extended-source case, focusing on same-time fluctuations and using harmonic polynomials as test functions; we are interested in strengthening his result to allow smooth test functions and to investigate correlations between fluctuations at different times.
	
	\paragraph{Acknowledgments.} I would like to thank Professor David Jerison and the MIT UROP+ program (organized by Slava Gerovitch) for making this project possible. I would like to especially thank David Jerison and Pu Yu (MIT Department of Mathematics) for their mentorship throughout. This research was supported in part by NSF Grant DMS 1500771.
	
	\bibliographystyle{alpha}
	\bibliography{../Bibliography}
	
\end{document}